\documentclass[10pt,reqno]{amsart}
\usepackage{amsmath,amsfonts,amssymb,amscd,amsthm,amsbsy,bbm, epsf,calc,enumerate,color,graphicx,verbatim,cite}
\usepackage{ifpdf}
\ifpdf
\pdfoutput=1
    \usepackage[colorlinks,citecolor=blue,linkcolor=blue]{hyperref}
\fi
\usepackage[shortalphabetic,initials]{amsrefs}
\usepackage{color}
\usepackage{datetime}

\newcounter{hours}\newcounter{minutes}

\makeatletter
\newtheorem*{rep@theorem}{\rep@title}
\newcommand{\newreptheorem}[2]{%
\newenvironment{rep#1}[1]{%
 \def\rep@title{#2 \ref{##1}}%
 \begin{rep@theorem}}%
 {\end{rep@theorem}}}
\makeatother

\theoremstyle{theorem}
\newtheorem{definition}{Definition}[section]
\newtheorem{theorem}{Theorem}[section]
\newreptheorem{theorem}{Theorem}

\newtheorem{lem}[theorem]{Lemma}
\newreptheorem{lem}{Lemma}
\newtheorem{cor}[theorem]{Corollary}

\theoremstyle{definition}

\theoremstyle{remark}                  

\def\R{{\mathbb R}}

\def\e{\varepsilon}

\def\cl{\bar}
\def\Rn{\mathbb R^n}
\def\vp{\varphi}

\newcommand{\abs}[1]{\left| #1 \right|}
\newcommand{\set}[1]{\left\{ {#1} \right\}}

\definecolor{darkgreen}{rgb}{0,0.4,0}
\definecolor{grey}{rgb}{0.5,0.5,0.5}

\newcommand {\no}{\noindent}

\numberwithin{equation}{section}

\begin{document}
\title{Free boundary problems for Tumor growth: a viscosity solutions approach}
\author{Inwon C. Kim}
\address{Department of Mathematics\\ UCLA\\
Los Angeles, CA 90095\\
USA}
\email[I.C. Kim]{ikim@math.ucla.edu}
\thanks{Inwon Kim is supported by the NSF grant DMS-1300445}

\author{Beno\^ \i t Perthame}
\address{ Sorbonne Universit\'es \\UPMC Univ Paris 06, UMR 7598 \\ Laboratoire Jacques-Louis Lions, F-75005\\ Paris, France}
\email[B. Perthame]{benoit.perthame@upmc.fr}\address{CNRS, UMR 7598, Laboratoire Jacques-Louis Lions, F-75005, Paris, France}
\address{INRIA-Paris-Rocquencourt, EPC MAMBA, Domaine de Voluceau, BP105, 78153 Le Chesnay Cedex, France} 
\thanks{Beno\^\i t Perthame  is supported by the ANR-13-BS01-0004 funded by the French Ministry of Research.}

\author{Panagiotis E. Souganidis}
\address{Department of Mathematics \\
University of Chicago \\
Chicago, IL 60637 \\
USA }
\email[P.E. Souganidis]{souganidis@math.uchicago.edu }
\thanks{Panagiotis Souganidis is supported by the NSF grant DMS-1266383.}

\begin{abstract}
The mathematical  modeling of tumor growth leads to singular ``stiff pressure law'' limits for porous medium equations with a source term. Such asymptotic problems give rise to 
free boundaries, which, in the absence of   active motion,  are generalized Hele-Shaw flows.
%
In this note we use viscosity solutions methods to study limits for porous medium-type equations with active motion. We prove the uniform convergence of the density under fairly general assumptions on the initial data, thus improving existing results. We also obtain some additional information/regularity about the propagating interfaces, which, in view of the discontinuities, can  nucleate and, thus, change topological type. 
The main tool is the construction of local, smooth, radial solutions which serve as barriers for the existence and uniqueness results as well as  to quantify the speed of propagation of the free boundary propagation.
\end{abstract}

\maketitle


{\bf AMS Class. No:}  35K55; 35B25; 35D40; 76D27; 

\section{Introduction}

\no Motivated by models of tumor growth (see for instance the survey papers by Friedman \cite{Friedman_survey}, Lowengrub {\em et al} \;Ê\cite{lowengrub_survey}) and extending  Perthame, Quiros and Vazquez \cite{PQV}, in a  recent paper  Perthame, Quiros, Tang and Vauchelet \cite{PQTV} studied  the limiting behavior, as $m\to\infty$, of the solution (density) $\rho_m$ of  the porous medium diffusion equation  (pme for short)
\begin{equation}\label{pme}
\rho_{m,t} - \Delta \rho_m^m -\nu\Delta \rho_m = \rho_m G(p_m) \  \hbox{ in }  \ Q_T:= \Omega \times (0,T),
\end{equation}
where $\Omega$ is either a bounded domain in $\R^n$ or $\Omega=\R^n$, with boundary and initial conditions 
\begin{equation}\label{pmeboundary}
\rho_m= \rho_L \  \ \hbox{ in }  \  \partial_p Q_T:=\partial \Omega\times [0,T) \  \text{and} \  
\rho_m=\rho_{0,m} \ \text{on} \  \Omega \times \{0\},
\end{equation}satisfying  
\begin{equation}\label{pmeboundary2}
0 \leq \rho_L<1  \ \text{and } \  0\leq \rho_0\leq 1 \ \text{ if $\Omega$ is bounded  or }  \rho_0 \in L^1(\R^n) \ \text{ if $\Omega=\R^n$}. 
\end{equation}
Here 
\begin{equation}\label{pressure}
p_m:= \frac{m}{m-1}\rho_m^{m-1}
\end{equation}
is the pressure,  $\nu >0$, and 
$G:\R \to \R$  is a smooth function, which describes the cell multiplication,  satisfying 
\begin{equation}\label{g}
G(p_M) = 0 \  \hbox{for some } \ p_M>0 \ \ \text{and} \ \  G' <0. 
\end{equation}


\no Using the pressure variable,  \eqref{pme}  can be rewritten as 
$$
\rho_{m,t} - \text{div} ( \rho_m Dp_m -\nu D\rho_m) = \rho_m G(p_m),
$$
a form which represents better the mechanical interpretation of the model  with $v_m:= -Dp_m$ the tissue bulk velocity according to Darcy's law. 
\smallskip

\no Note that, if, as $m \to \infty$,  the $p_m$'s and $\rho_m$'s converge respectively to $p$ and $\rho$, then $p$ will be nonzero  only where $\rho=1$. This indicates, that, in the limit,  a phase transition may take place with an evolving free boundary  between the tumor region (the support of $p$) and the pre-tumor zone (the support of $1-\rho$). The convergence,  as $m\to\infty$,  of the $p_m$'s and $\rho_m$'s  has already been investigated using a distributional solution approach in \cite{PQTV}. 

\no  Here we study the asymptotic behavior  of $p_m$ and $\rho_m$  in the limit  $m\to\infty$ using  viscosity solutions. This yields a different description of the limit problem,  allows for more general initial data and yields pointwise information about the free boundary evolution, uniform convergence results  as well as some quantified statements about the speed of propagation of the tumor zone.

\no In order to state the result it is necessary to introduce the limit problem which we derive next formally following \cite{PQTV}.  We use the auxiliary variable
\begin{equation}\label{u_m}
u_m:= -\rho_m^m+\nu(1-\rho_m),
\end{equation}
which is close to $-p_m + \nu(1-\rho_m)$ for large $m$,  and, recalling that \eqref{pme} can also be written as 
\begin{equation}\label{eqn:p}
p_{m,t} - (m-1) p_m \Delta p - |D p_m|^2-\nu\Delta p_m = (m-1) p_m G(p_m) - \nu\frac{m-2}{m-1}\frac{|D p_m|^2}{p},
\end{equation}
we find
that $u_m$ satisfies
\begin{equation}\label{eqn:um}
[b_m(u_m)]_t - \nu \Delta u_m = - \nu \rho_mG(p_m) \ \  \text{with} \ \  b_m'(u_m) =\frac{\nu}{m\rho_m^{m-1}+\nu}.
\end{equation}

\no Assume next that, as $m \to \infty$, the $\rho_m$'s, $p_m$'s and $u_m$'s  converge respectively  to $\rho$, $p$ and $u$. Then, formally, we find 
\begin{equation}\label{uinfty}
p = u^- , \quad  \nu(1-\rho)= u^+  \qquad\hbox{ where } u^+:= \max(u,0)  \hbox{ and } u^-:=-\min(u,0).
\end{equation}

\no Letting $m\to\infty$ in \eqref{eqn:p} and noting that $\{p>0\} = \{\rho=1\}$ yields that $p$ and $\rho$ solve respectively
  \begin{equation}\label{elliptic}
  -\Delta p(\cdot,t) = G(p)(\cdot,t) \ \text{ in } \  \Omega(t):=\{p(\cdot,t)>0\},
  \end{equation}
and
$$
\rho_t - \nu\Delta \rho = \rho G(0)  \ \text{ in} \  \{\rho<1\};
$$  
this last equation which can be rewritten 
 in terms of $\tilde{\rho}:=\nu(1-\rho)$ as
  \begin{equation}\label{parabolic}
  \tilde{\rho}_t - \nu\Delta\tilde{\rho} = -\nu\rho G(0) = (\tilde{\rho}-\nu)G(0) \  \hbox{ in } \ \R^n-\Omega(t).
  \end{equation}

\no Combining  \eqref{uinfty}, \eqref{elliptic} and \eqref{parabolic} the limiting problem can be recast with $b(u):= u_+$ as the elliptic-parabolic equation
\begin{equation}\label{parabolic_elliptic}
b(u)_t - \nu\Delta u = (b(u)-\nu)G(u^-) \  \hbox{ in } \ Q_T.
\end{equation}

\no In view of the Lipschitz continuity of $b$, \eqref{parabolic_elliptic} yields  the free boundary condition 
\begin{equation}\label{no_jump}
  \partial_\eta u^+ = \partial_\eta u^- \hbox{ on } \partial\Omega(t), 
  \end{equation}
  where $\eta$ is the (outward) normal at $(x,t)\in\partial\Omega(t)$, which provides an implicit motion law for the free boundary $\partial\{u>0\}$. 
\smallskip
  
\no  This problem but without the growth term  has been studied by Alt-Luckhaus \cite{AL} in the weak (duality) setting, by Carrillo \cite{CAR} in the weak (entropy) setting and also by  Kim-Pozar \cite{KP} using viscosity solutions. The right hand side, however,  plays an important role because it can generate pressure nucleation, thus generating a change of topology in the free boundary (in a zone where $p=0$, $\rho$ can grow and reach $\rho =1$ thus generating a new island where $p>0$).  Viscosity solutions for porous medium equations were introduced by Caffarelli and Vazquez \cite{CafVaz}, see also the book by Vazquez \cite{vazquez_book}, and later by Kim \cite{Kim2003} for the Hele-Shaw problem, for the Stefan problem and additional references the reader can refer to the book by Caffarelli and Salsa \cite{CafSalsa}.

\smallskip

\no Next we introduce the precise assumptions on the initial data $\rho_{m,0}$. We assume that  

\begin{equation}\label{initial_intro}
\text{ there exists $M_0 >0$ such that $p_{m,0}:= \frac{m}{m-1}(\rho_{m,0})^{m-1} \leq M_0$,}
\end{equation}
and, as $m\to \infty$,  
\begin{equation}\label{initial_intro2}
\begin{cases}
\text{$\rho_{m,0} \to \rho_0$  uniformly, where $\rho_0:\Omega \to [0,1]$  is continuous and} \\[.8mm]
\text{$\Omega_0:= \{\rho_0 = 1\}$ is bounded domain with locally Lipschitz boundary.}
\end{cases}
\end{equation}
\smallskip
and, in addition, if  $\Omega=\R^n$ that 
\begin{equation}\label{ubdd}
\rho_0\in L^1(\R^n). 
\end{equation}

\no The main result of the paper is: 

\begin{theorem}\label{thm:main}
Let $\rho_m$ solve \eqref{pme} with  data satisfying  \eqref{pmeboundary2}, \eqref{initial_intro}, \eqref{initial_intro2}, and  \eqref{ubdd}, if $\Omega=\R^n$, 
and define $ 
u_0(x):= -w \chi_{\Omega_0} + \nu(1-\rho_0)\chi_{\Omega_0^c},
$
where $w$ be the unique solution of  $-\Delta w = G(w)$ in $\Omega_0 \ \text{and } \ w=0 \ \text{on} \ \partial \Omega_0$.
Then, for all   $T>0$ and as $m\to \infty$, the $\rho_m$'s  converge  to $1-\nu^{-1}b(u)$ uniformly in $Q_T$,   where $u$ is the unique viscosity solution of \eqref{parabolic_elliptic}, with the initial and boundary data $u_0$ and $\nu(1-\rho_L)$ respectively.
\end{theorem}

\no We remark that the $p_m$'s converge  uniformly  to $u_-$ as long as $u_-$ is continuous. It turns out, however,  that $u_-$  may be  discontinuous in time. This is  due to the fact that (the positive part) $b(u)$ solves a parabolic equation with a sink term, which means that $b(u)$ can decrease to zero in the interior of its positive phase and nucleate a negative phase. Once the negative phase is created, the elliptic equation that needs to be satisfied in the negative phase leads to the jump discontinuity of $u_-$ over time. We refer to Section~\ref{sec:last} for a discussion about the propagation of the elliptic and parabolic phases. The Figure~\ref{fig:HSV}  illustrate the time discontinuities in $p_m$, and Figure~\ref{fig:HS} shows and additional discontinuity in the $\rho_m$'s when viscosity is neglected (Hele-Shaw problem). 
\smallskip

\begin{figure}[h]
	\centering
	\includegraphics[width=.26\textwidth]{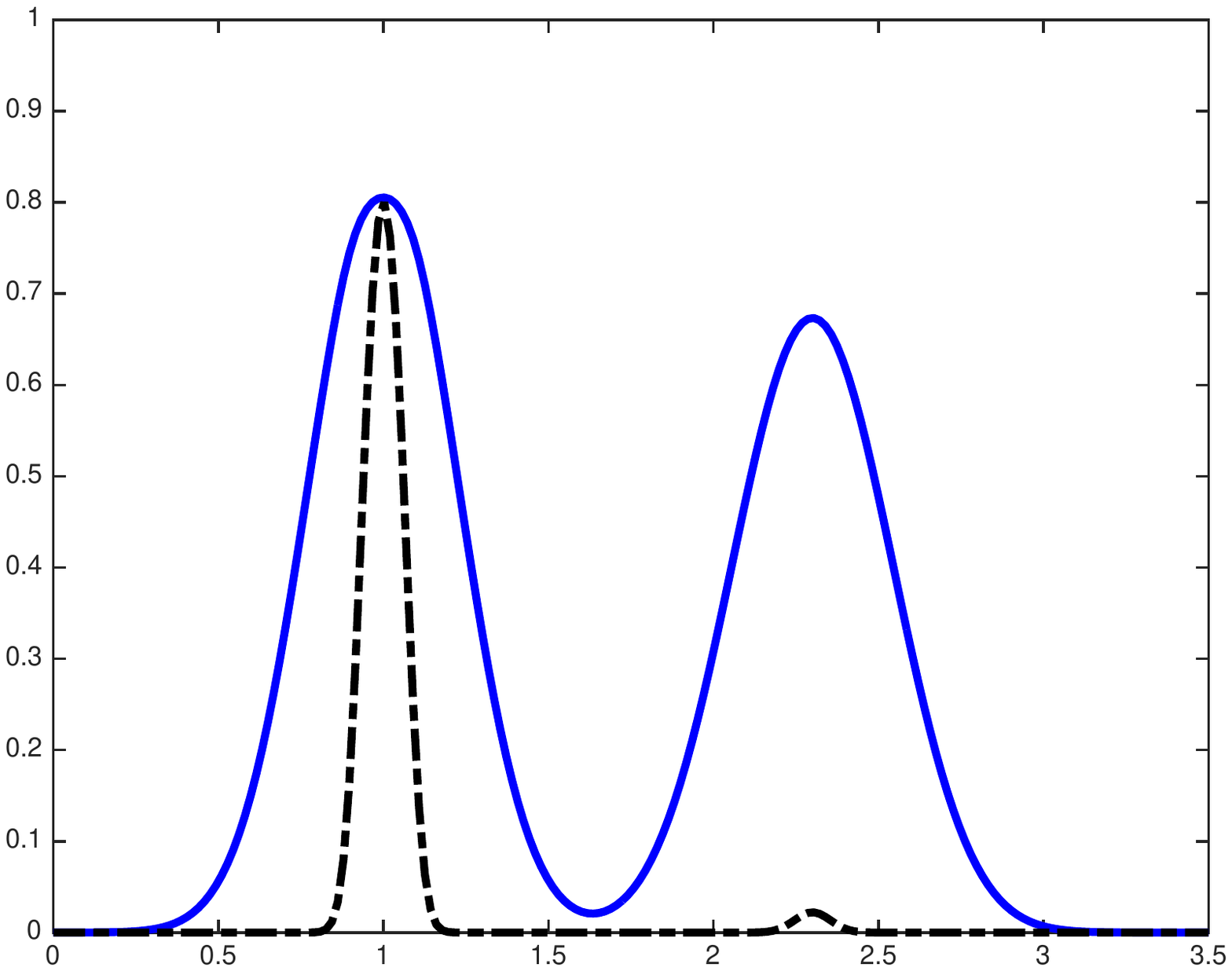}\hspace{-20pt}
	\includegraphics[width=.26\textwidth]{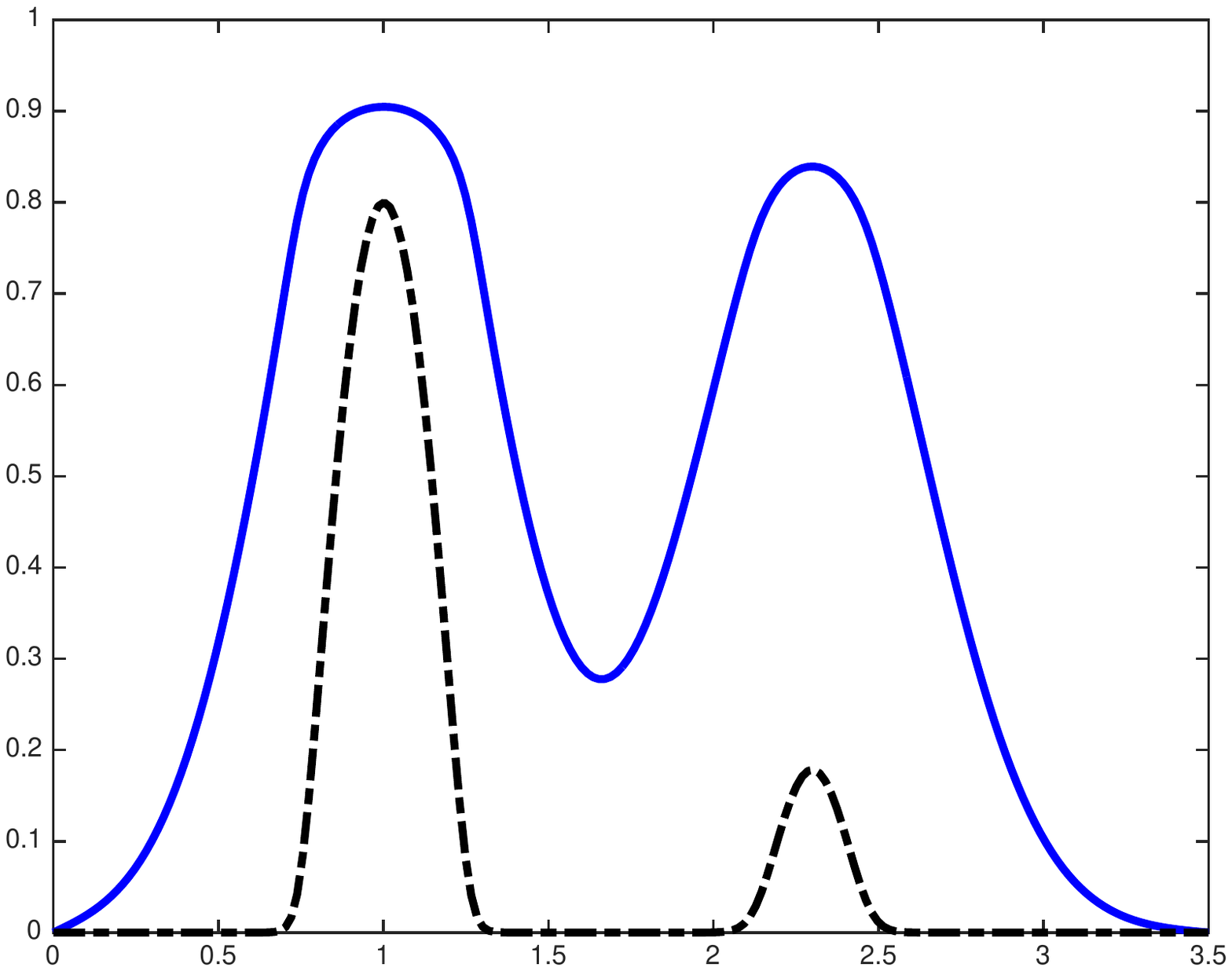}\hspace{-20pt}
	\includegraphics[width=.26\textwidth]{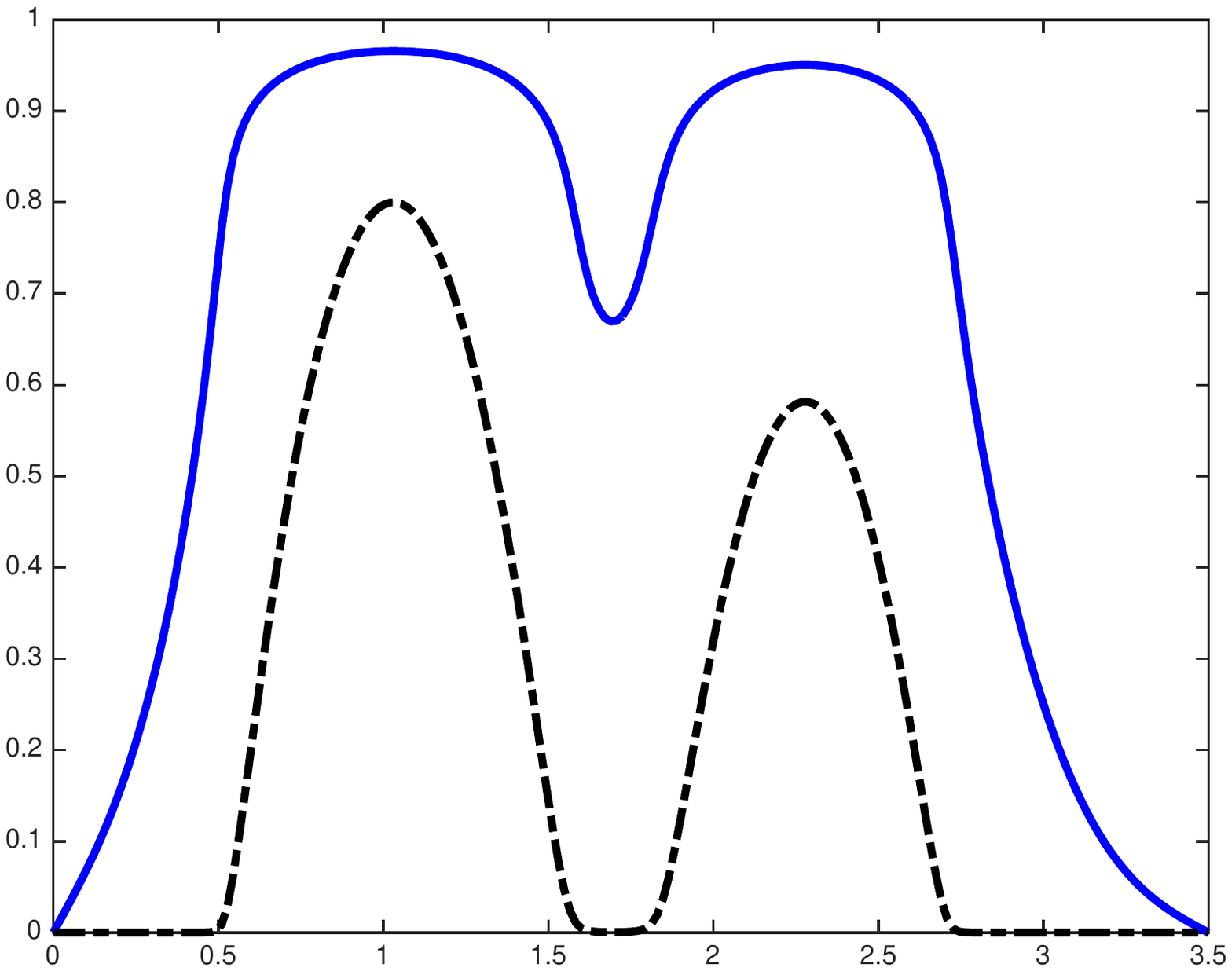}\hspace{-20pt}
	\includegraphics[width=.26\textwidth]{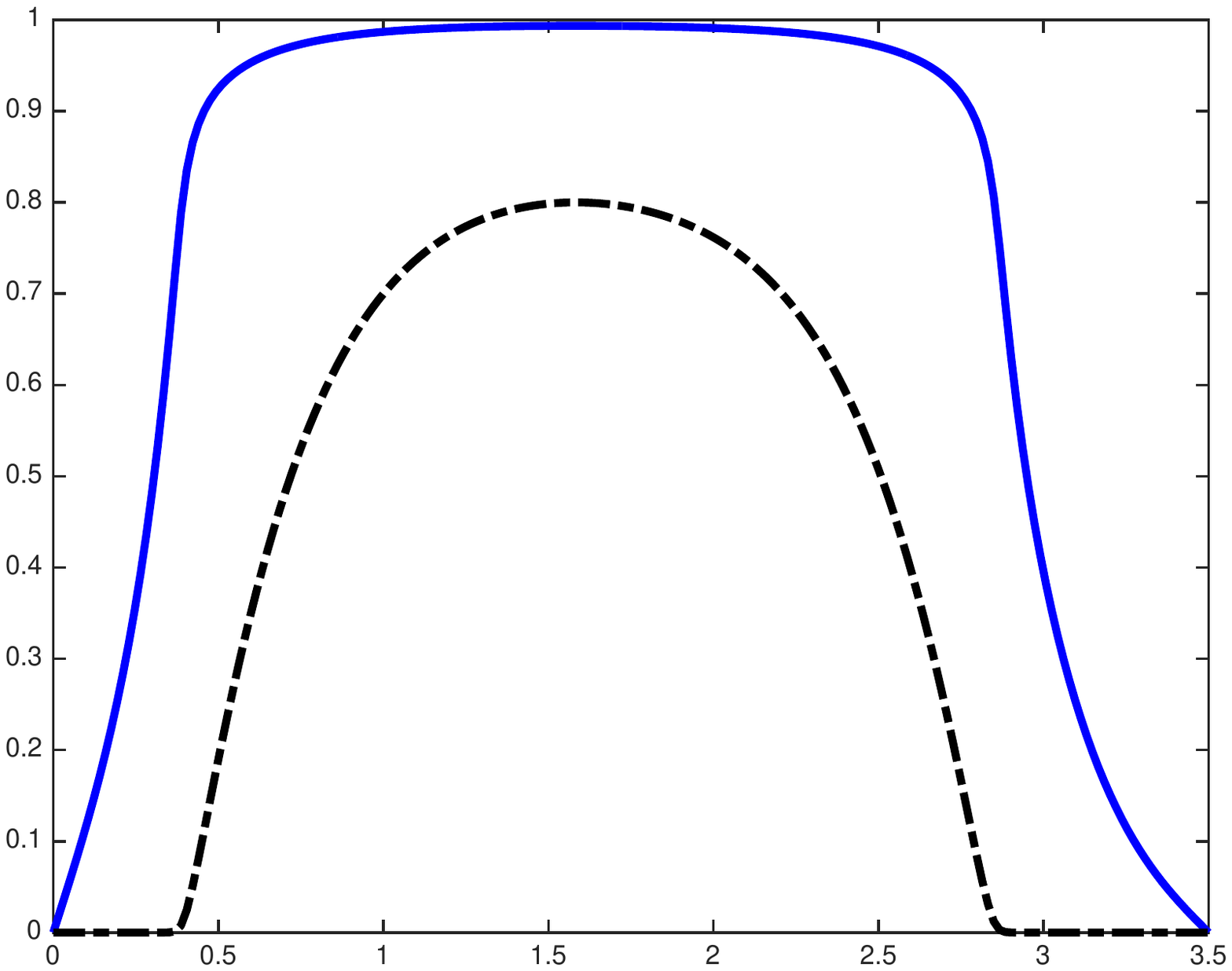}
\vspace{-18mm}		
	\caption{Snapshots of the evolution of the density $\rho_m$ (blue solid line) and pressure $p_m$ (black discontinuous line). The parameters are $m=20$, $\nu=.5$. These figures illustrate how the pressure profile is building up when density approaches one. Between the last two pictures, the density has continuously reached $\rho=1$ in the center, while $p$ has jumped discontinuously.}
	\label{fig:HSV}
\end{figure}

\begin{figure}[h]
	\centering
	\includegraphics[width=.26\textwidth]{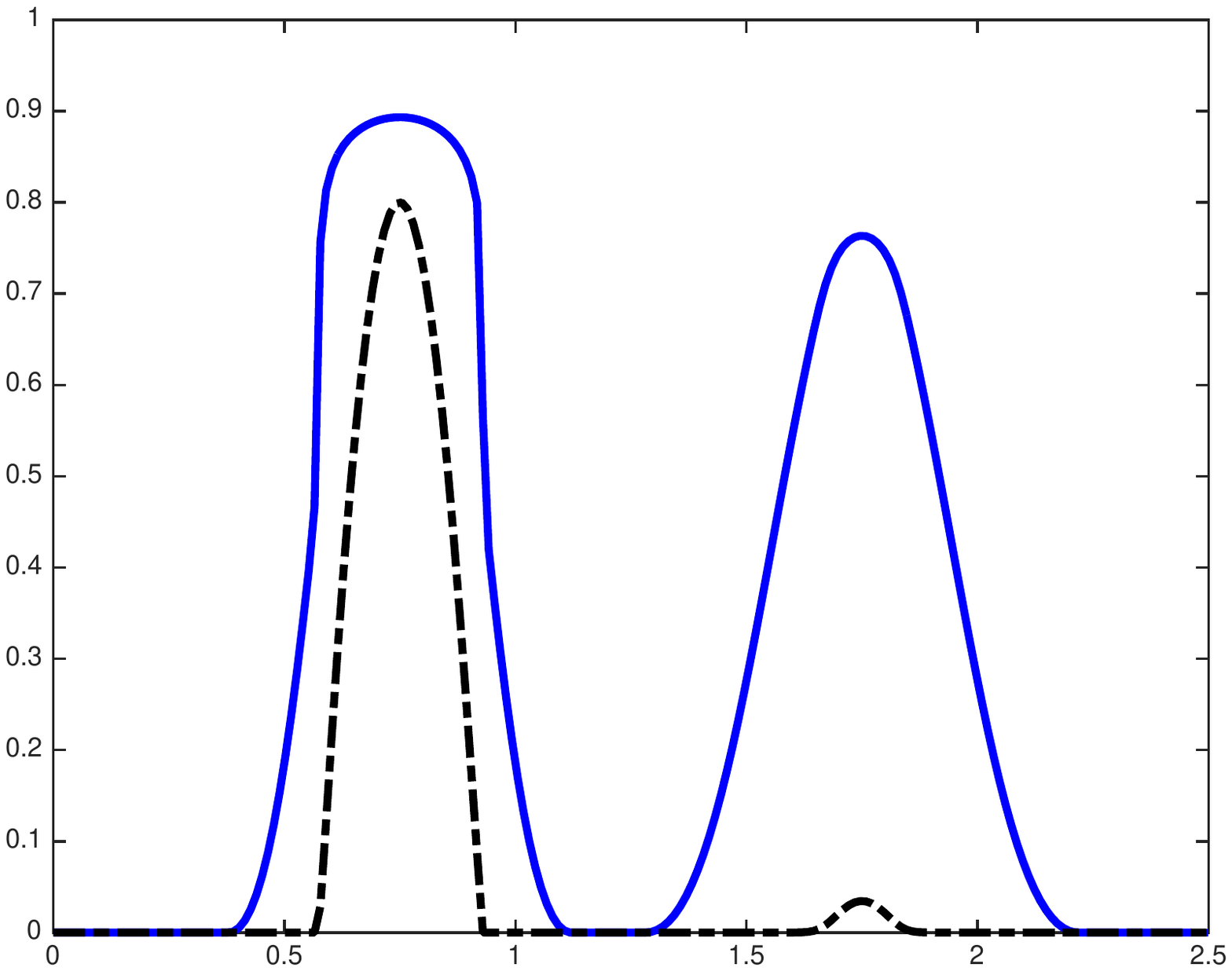}\hspace{-20pt}
	\includegraphics[width=.26\textwidth]{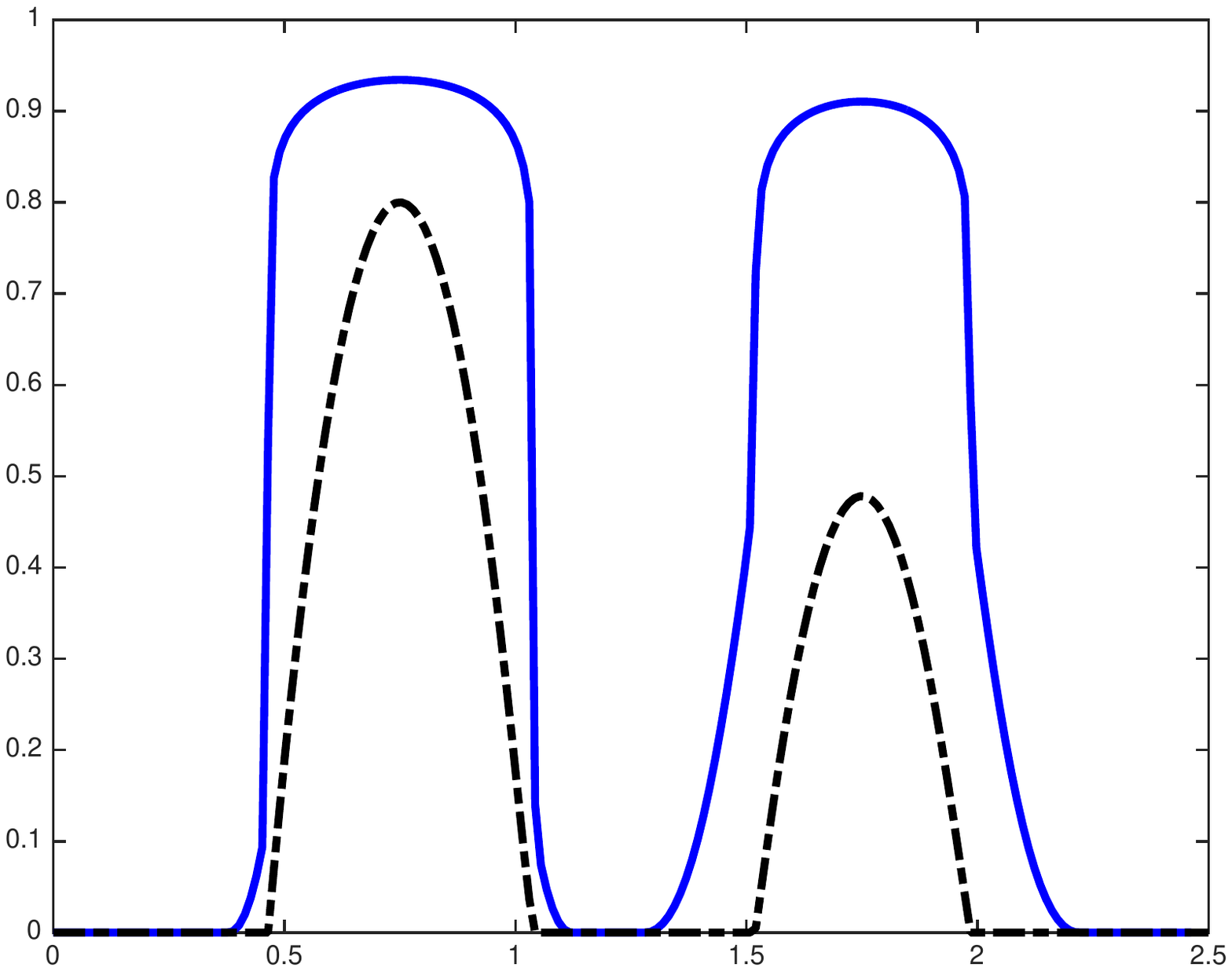}\hspace{-20pt}
	\includegraphics[width=.26\textwidth]{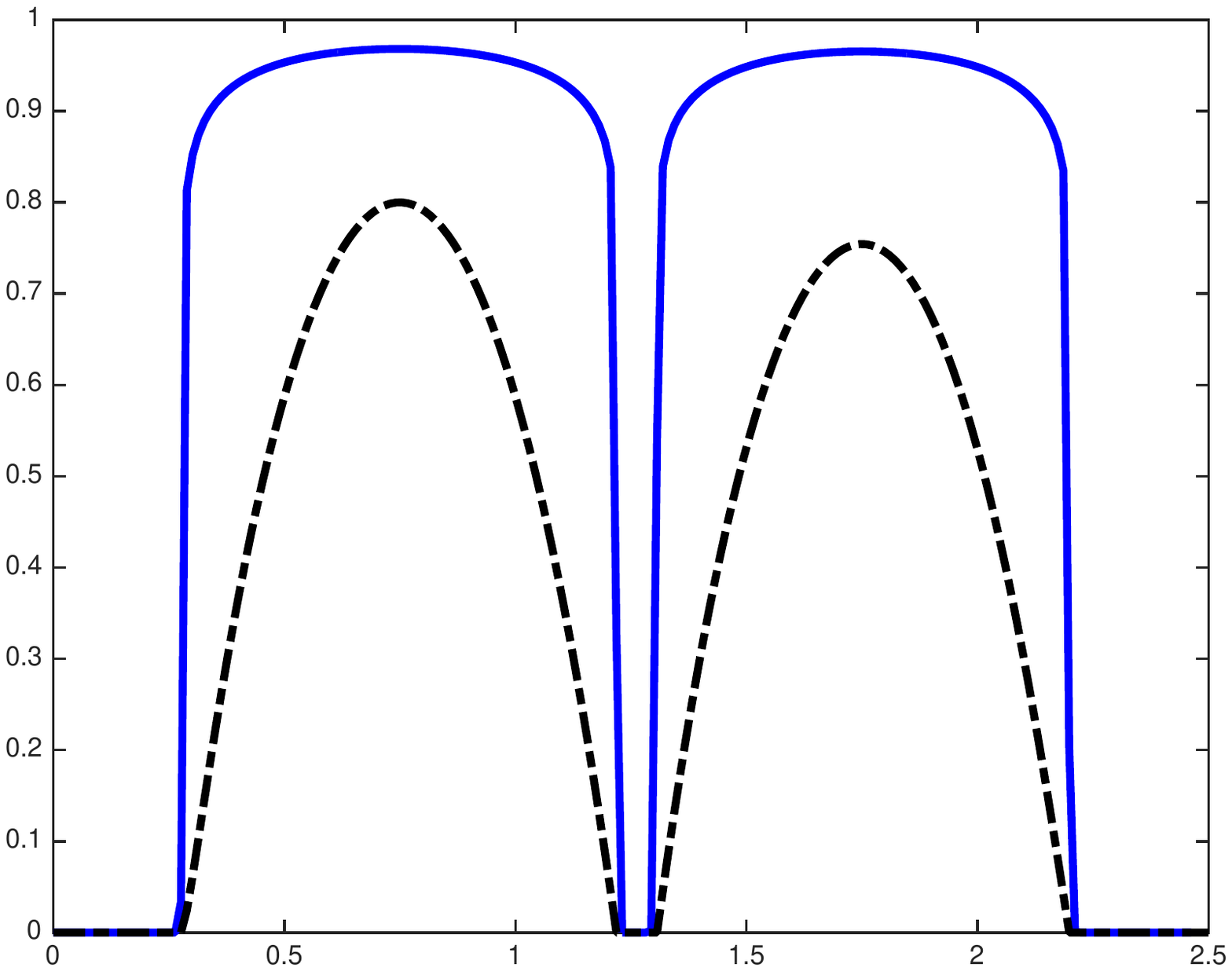}\hspace{-20pt}
	\includegraphics[width=.26\textwidth]{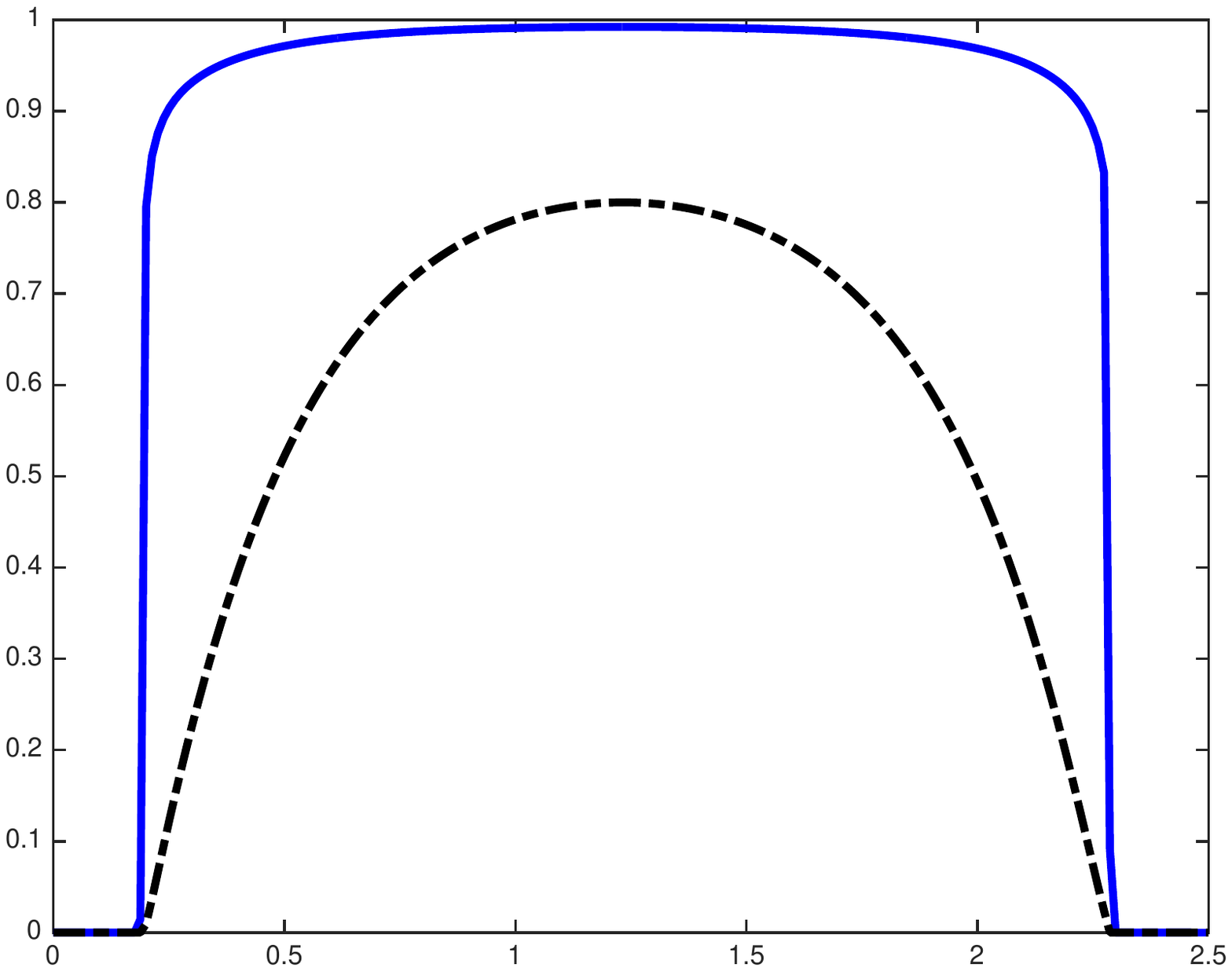}
\vspace{-18mm}		
	\caption{Snapshots of the evolution of the density $\rho_m$ (blue solid line) and pressure $p_m$ (black discontinuous line). The parameters are $m=20$, $\nu=0$ (not treated in this paper). Compared to Figure~\ref{fig:HSV}, we observe also a discontinuous behaviour of the density.}
	\label{fig:HS}
\end{figure}

\no Theorem \ref{thm:main} is proved by showing first that the $u_m$'s   converge to $u$ as $m\to\infty$ (see Corollary~\ref{convergence}). The convergence results for the $\rho_m$'s then follow from the definition of $u_m$ and the continuity properties of $b(u)$. 
  
\smallskip
 
\no As already mentioned above we use viscosity solutions which are based on appropriate choices of test functions. In the problem at hand these will be radial smooth solutions which are defined in Section~\ref{sec:limit}.  The main step in the proof of Theorem \ref{thm:main} is the following: 
 
 \begin{theorem} \label{thm:radial}
Assume that  $\phi$ is a classical radial solution of \eqref{parabolic_elliptic}. Then there exists a family of radial sub- and solutions super-solutions of \eqref{eqn:um}  converging, as $m\to \infty$,  to $\phi$. 
\end{theorem}
 
 
 
\subsection*{Organization of the paper.} 
In Section~\ref{sec:limit} we introduce the notion of viscosity solution of the limit problem and discuss its existence and uniqueness which rely on the results of \cite{AL} and \cite{KP}. In Section~\ref{sec:radial} we prove Theorem~\ref{thm:radial}. We used the results of Section~\ref{sec:limit} and Section~\ref{sec:radial} to prove Theorem \ref{thm:main} in Section~\ref{sec:general}. In Section~\ref{sec:last}  we present estimates on the speed of propagation  of the interface. In the Appendix we touch upon the proof of the comparison for viscosity solutions to the limit problem.

\subsection*{Notation and Terminology}  A nonempty set $E \subset \R^n\times \R$ is called a parabolic neighborhood  if $E = U \cap \big(\R^n \times (0, \tau] \big)$ for some open set $U \subset \R^n\times \R$ and some $\tau \in(0,\infty)$ and  $\partial_P E := \bar{E} \setminus E$ is  the parabolic boundary of $E$. Whenever we refer to a parabolic neighborhood, we assume $U$ and $\tau$ are known. For  an open subset $A$ of $\R^n \times \R$ and a  continuous function $v: \bar A \to \R$, where $\bar A$ is the closure of $A$, $\{v >0\}, \{v=0\}$ and $\{v<0\}$ are respectively the sets $\{y\in \bar A: v(y) >0\}, \{y\in \bar A: v(y)=0\}$ and $\{y\in \bar A: v(y)<0\}$.  For a set $A$, we write $\chi_A$ for its characteristic function.  If $A, B$ are subsets of $R^k$, $A+B:=\{a+b: a\in A, b \in B\}$. Given a set $U\subset \R^k$, $USC(U)$ and $LSC(U)$ are respectively the sets of upper and lower semicontinuous functions on $U$; if we are dealing with bounded semicontinuous functions we write $BUSC(U)$ and $BLSC(U)$.  If $u: U\to \R$ is bounded, $u^{\star, \bar U}$ and $u_{\star, \bar U}$ denote respectively its  upper and lower-semicontinuous envelope. 
We write $B_R(x)$ for the open ball in $\R^n$ centered  $x\in \R^n$ with radius $R>0$. The duality map between a Banach space and its dual is denoted by $<\cdot, \cdot>$. If  $A$ is open subset of $\R^n \times \R$,  $C_{x,t}^{2,1}(A)$ denotes the space of functions on $A$ which are continuously twice differentiable in $x$ and continuously differentiable in $t$; similarly, if $B$ is an open subset of $\R^k$, $C^l_x(B)$ denotes the space of $l$-times  continuously differentiable functions on $B$; when there is no confusion we omit the dependence on $A$ and $B$. We say that a constant is dimensional if it depends only on the dimension.  We denote by $D\phi (\xi,\tau)$ the gradient at $(\xi,\tau)$ of  $\phi:U\to \R$, where $U$ is an open subset of $\R^n \times \R$ and, $D\phi^{\pm}(\xi,\tau):= \lim_{\substack{(x,t) \to (\xi,\tau)\\(x,t) \in \set{\pm \vp > 0}}} D\vp(x,t)$. For a family of functions $f_m :\Omega\times (0,\infty) \to \R$, ${\liminf}_*f_m$ and ${\limsup}^*f_m$  denote the lower and upper semi-continuous limits, that is 
$
{\liminf}_*f_m(x,t):= \lim_{r\to 0}\inf\{  f_m(x+y, t+s) : (x+y,t+s)\in\bar{\Omega}\times (0,\infty), |y|+|s|\leq r, m\geq r^{-1}\},
$
and
$
{\limsup}^*f_m(x,t):= \lim_{r\to 0}  \sup\{  f_m(x+y, t+s) : (x+y,t+s)\in\bar{\Omega}\times (0,\infty), |y|+|s|\leq r, m\geq r^{-1}\}.
$

\section{Viscosity and weak solutions of the limit problem}
\label{sec:limit}
\no We introduce the notions of viscosity and weak (distributional) solutions to the limit problem 
\begin{equation}\label{limitproblem}
\begin{cases}
b(u)_t - \nu\Delta u = (b(u)-\nu)G(u^-) \  \hbox{ in } \ Q_T,\\[1mm]
u=u_0 \ \text{on} \ \Omega \times \{0\},\\[1mm]
u=g \ \text{on} \ \partial \Omega \times [0,T],
\end{cases}
\end{equation}
and we discuss their existence and uniqueness. While we rely on the theory of weak solutions  to obtain uniqueness results, viscosity solutions are used to prove  pointwise convergence results in Section~~\ref{sec:general}  as well as to derive information about  the evolution of parabolic and elliptic phases in Section~\ref{sec:last}.

\subsection*{Viscosity solutions: definition and comparison principle.} Following  \cite{KP} we define viscosity solutions of \eqref{limitproblem}. We begin by introducing the class of  test functions which are classical sub-and super-solutions of the problem with the specific properties stated in the next definition.

\begin{definition}[Classical sub- and super-solutions]
\label{de:classicalSubsolution}
$E$ be a parabolic neighborhood. Then $\vp \in C(\cl{E})$ is  a classical subsolution of  \eqref{parabolic_elliptic} in  $E$ if
\vskip.025in
\no (i)~$\vp \in C^{2,1}_{x,t}(\overline{\{\vp > 0\}})\cap C^{2,1}_{x,t}(\overline{\{\vp < 0\}})$,
\vskip.025in
\no (ii)~ $\set{\vp = 0} \subset \partial\set{\vp > 0} \cap \partial\set{\vp < 0}$ and $\abs{D\vp^\pm}  > 0$ on $\set{\vp = 0}$,
\vskip.025in
\no (iii)~ $b(\vp)_t - \nu\Delta \vp \leq  (b(\vp) - \nu)G(\vp^+)$ on $\set{\vp > 0}$ and $\set{\vp < 0}$, and 
\vskip.025in
\no (iv)~$\abs{D\vp^+} > \abs{D\vp^-}$ on $\set{\vp = 0}$;

\no $\vp$ is a strict classical subsolution if the inequalities in (iii) and (iv) are strict. 
Supersolutions are defined similarly by reversing the inequalities in (iii) and (iv). 
\end{definition}

\no Next we define the viscosity solutions to \eqref{parabolic_elliptic}. We begin with the notions of (semicontinuous) sub-and super-solutions.

\begin{definition}[Viscosity sub- and super-solutions]
\label{de:viscositySubsolution}

We say $u \in BUSC(\overline {Q_T})$ (resp.  $u \in BUSC(\overline {Q_T})$  is a viscosity subsolution (resp. supersolution) to \eqref{limitproblem} in $Q_T$,  if 
\vskip.025in
\no (i)~$u(\cdot, 0) \leq u_0$ (respectively  $u(\cdot, 0) \geq u_0$) on $\Omega \times \{0\}$, $u \leq g$ (resp. $u \geq g$) on $\partial \Omega \times [0,T]$, 
\vskip.025in
\no (ii)~ $u < \vp$ (resp. $u > \vp$) on $E$ for any strict classical subsolution (resp. supersolution) $\vp$ on any parabolic neighborhood $E \subset Q_T$ for which $u < \vp$ (resp. $u > \vp$) on $\partial_P E$. 
 
We say $u$ has initial data $u_0$ if $u$ converges uniformly to $u_0$ as $t\to 0$.   
\end{definition}

\no Given that we do not expect to have continuous but rather lower semicontinuous solutions to \eqref{limitproblem} (recall the earlier discussion about possible nucleation),  we define the notion of viscosity solution for such functions.

\begin{definition}[Viscosity solutions]
A function  $u \in BLSC(\overline {Q_T})$ is a viscosity solution of \eqref{limitproblem} in $Q_T$ if its upper semi-continuous envelope  $u^{*,\overline { Q_T}}$ is a viscosity subsolution in $Q_T$ and $u$ is a viscosity supersolution in $Q_T$.
\end{definition}

\no Next we introduce conditions on the data of \eqref{limitproblem}  that are necessary in order to have a well-posed theory. 
\smallskip

\no For the lateral boundary we assume, for simplicity, that 
\begin{equation}\label{lbc}
g \ \text{ is continuous and positive.}
\end{equation}

\no For the initial condition,  if $\Omega$ is bounded,  we assume that $u_0: \overline{\Omega} \to \R$ satisfies 
\begin{equation}\label{initial_1}
u_0\in C(\overline{\Omega}),  \ u_0<0 \hbox{ on } \partial\Omega \hbox{ and } -\Delta u_0 = G(u_0) \hbox { in } \{u_0 \geq 0\},
\end{equation} 
and 
\begin{equation}\label{initial_2}
\Gamma(u_0):=\partial\{u_0 \geq 0\}\hbox{ is locally a Lipschitz graph.}
\end{equation}

\no When  $\Omega=\R^n$, we require, in addition, that
\begin{equation}\label{cond_unbdd}
\{u(x,0)<0\} \subset B_R(x) \hbox{ for some }R>0 \hbox{ and } u(\cdot,0)-\nu\in L^1(\R^n).
\end{equation}

\no Note  that,  in terms of $p=u^+$ and $\rho = \nu^{-1}u_- +1$, the above conditions imply that the pressure phase is initially bounded and the density is integrable.

\no We recall that we use only strict sub-and super-solutions as test functions. 
As a result it is possible to narrow the choice of the parabolic neighborhood $E$ in Definition~\ref{de:viscositySubsolution} and to use, instead, parabolic cylinders of the form $Q' = \Omega' \times (t_1, t_2] \subset Q_T$, with  $\Omega' \subset \Omega$ having a smooth boundary.  
Indeed, suppose that $E \subset Q_T$ is a parabolic neighborhood, $\vp$ is a strict classical supersolution on $E$, $u < \vp$ on $\partial_P E$ and  $u \geq \vp$ at some point in $E$. Define $\tau:= \sup \set{s: u < \vp \text{ on } E \cap \set{t \leq s}} \in \R$. The set
$A:= \set{x : (x,\tau) \in E, u \geq \vp}$
is compact and, therefore, $d:= \text{distance}(A \times \set{\tau}, \partial_P E) > 0$. Consider the parabolic cylinder $Q' = (A + B_{d/2}) \times (\tau - d, \tau]$. Clearly $Q' \subset E$ and $u < \vp$ on $\partial_P Q'$; the boundary of $A + B_{d/2}$ can be easily regularized.


\smallskip

\no The proof of the comparison principle follows along the lines of the one in \cite{KP} the only difference being  the contribution of the source term $(b(u)+\nu)G(u)$, which is not present in the equation considered in \cite{KP}. The needed modifications are presented  in the Appendix.
Here we note, that, although, as discussed earlier, the source term in the negative (elliptic) phase of $u$ may give rise to nucleations, this does not
hinder the argument  in \cite{KP}. Indeed nucleations  can not occur at the contact point of two ``regularized'' solutions which are  initially strictly ordered (see Appendix). Hence we have the following theorem, which corresponds to Theorem 3.1 of \cite{KP}.

\begin{theorem}\label{thm:cp}
Assume \eqref{lbc}, \eqref{initial_1}, \eqref{initial_2} and, if $\Omega=\R^n$, \eqref{cond_unbdd} and let  $u \in BUSC(\overline {Q_T})$ and $v\in BLSC(\overline {Q_T})$ be respectively a sub- and a super-solution of \eqref{parabolic_elliptic} in $Q_T$ for some $T>0$. If $u< v$ on the parabolic boundary of $Q_T$, then $u<v$ in $Q_T$.
\end{theorem}

\no Note that in the above theorem we require  that $u$ and $v$ are strictly separated  on the parabolic boundary of $Q_T$. For this reason, Theorem~\ref{thm:cp}  yields,  for  given data, only   the maximal and minimal viscosity solutions. Removing the condition of strict separation  remains an interesting (and important) open question for most of free boundary problems. For the problem at hand, we are able to borrow the theory of  regular weak solutions of \cite{AL} to study the  stability and uniqueness properties of the viscosity solutions.

\subsection*{Weak solutions: existence and comparison principle} We recall the notion of regular weak solution to \eqref{limitproblem} and  then state the theorems in \cite{AL} which concern the existence and stability of such solutions. The difference in the proofs due to the source terms are discussed 
in the Appendix. We note that we do not try to recall the full generality of \cite{AL} but we modify definitions and statements to apply to the specific problem we study here.   Finally,  we remark that \cite{AL} only considers bounded domains $\Omega$.

\smallskip

\no We begin with the definitions of weak and regular solutions. The former is a notion of solutions based on duality (integration by parts) while the second, as the name indicates, requires  more regularity (in time).

\begin{definition}[Weak solutions] $u \in  L^{2}(0,T;H^{1}(\Omega))$ with $b(u)\in L^\infty(0,T;L^1(\Omega)) \ \text{ and } \ \partial_t b(u)\in L^2(0,t; H^{-1}(\Omega))$ is a weak solution to 
to \eqref{parabolic_elliptic} in $Q_T$ if, 
for every test function $\zeta\in L^2(0,T;H^{1}_0(\Omega))$,  
\begin{equation}\label{weak}
\int_0^T<\partial_t b(u), \zeta> dt + \int_0^T\int_{\Omega} D u\cdot D\zeta dxdt = \int_0^T \int_{\Omega} (b(u)-\nu)G(u^-)\zeta dxdt.
\end{equation}
Weak sub- and super-solutions are defined with the corresponding inequalities replacing the equality above.
\end{definition}
\begin{definition}[Regular weak solutions] A weak solution $u$ is regular if $\partial_t b(u) \in L^2(Q_T)$.
\end{definition}

\no The following corresponds to Theorem 2.2 of \cite{AL}.

\begin{theorem}\label{comparison_AL}
Assume that the bounded domain $\Omega$ has Lipschitz boundary. If $u$ and $v$ are respectively a weak subsolution and supersolution of \eqref{parabolic_elliptic} in $Q_T$ 
and, in addition,  $\partial_t (b(u)-b(v))\in L^1(Q_T)$ and  $u \leq v$ on the parabolic boundary of $Q_T$, then $u\leq v$ in $Q_T$.
\end{theorem}
 
\no The proof  parallels the one  of Theorem 2.2 in \cite{AL}. It  is rather simple, and,  for the reader's convenience,  we present it next.
\begin{proof}
Fix $\delta>0$ and let $\psi(z):=\min [1,\max(z/\delta,0)]$. Applying the definition to $u$ and $v$ with $\zeta:= \psi(u-v)$ and using the fact that  $G$ is decreasing and $u_-<v_-$ if $v<u$, we find, for some $C>0$, 
$$
\begin{array}{l}
\int_0^t\int_{\Omega} \partial_t(b(u)-b(v))\psi(v-u) dxdt + \frac{C}{\delta}\int_0^t \int_{\Omega} \chi_{\{0<u-v<\delta\}} |D(u-v)|^2dxdt\\ \\   \leq\int_0^t\int_{\Omega} [\chi_{\{0<u-v<\delta\}}(b(u)-\nu) G(u_-) - (b(v)-\nu)G(v_ -)]dxdt\ \\
=\int_0^t \int_{\Omega}[\chi_{\{0<u-v<\delta\}} (b(u)-b(v))G(0) +\nu(-G(u_-) +G(v_-))]dxdt\\  \\
\leq \int_0^t  \int_{\Omega} (b(u)-b(v))_+ G(0) dxdt.
\end{array} 
$$

\no Next we let  $\delta\to 0$. The parabolic term on the left converges to $\int_\Omega (b(u)-b(v))_+(\cdot,t) dx$, and, thus,  Gronwall's lemma yields that $b(u) \leq b(v)$. 
\smallskip

\no Using this last information in the previous inequality, we find that,  for any $\delta>0$, $b(u-v)=0$ in $\{0<u-v<\delta\}$, which, in view of the fact that  $u\leq v$ on the parabolic boundary of $Q_T$, yields the claimed comparison.
\end{proof}
\no The next results assert the existence and uniqueness of regular weak solutions.

\begin{theorem}\label{existence}
Assume that the bounded domain $\Omega$ has Lipschitz boundary. Then there exists a unique regular weak solution $u$ in $Q_T$ for a given boundary data $u_D\in H^1(0,T;H^1(\Omega))$. Moreover, $b(u)$ is continuous.
\end{theorem}

\no  The uniqueness follows from Theorem~\ref{comparison_AL} while the continuity of $b(u)$ is a consequence of  the main result of Di Benedetto and Gariepy \cite{DG}. The existence  is based on the Galerkin approximation and parallels that of Theorem 2.3 in \cite{AL} where we refer for the details.

\smallskip


\no Weak solutions to \eqref{parabolic_elliptic} are more flexible than viscosity solutions with respect to  the range of the boundary data. On the other hand,  we do not see how to prove directly the pointwise convergence of solutions of \eqref{eqn:um} to the solutions of \eqref{parabolic_elliptic} without going through the viscosity solutions framework. Indeed the convergence proof depends on the facts that we allow semi-continuity for viscosity sub- and super-solutions, and that, even with  such low regularity,  there is a  comparison principle. 

\subsection*{Uniqueness and existence of viscosity solutions for bounded domains} We discuss now the uniqueness for the viscosity solutions and we argue separately depending on whether $\Omega$ is bounded or not.. 


\subsection*{Bounded domain} The first step to establish the  uniqueness is that viscosity solutions evolve continuously from the  initial data. This is quantified in terms of the growth in time of the distance between the parabolic phase and $\partial \Omega \setminus \Omega_0$, the initial position, where  $\partial \Omega_0:=\{u_0<0\}$.
\begin{lem}\label{initial0}
Assume \eqref{g}, \eqref{lbc}, \eqref{initial_1} and \eqref{initial_2} and  let $u$ be a viscosity solution of \eqref{limitproblem}. 
There exists some sufficiently small $t_0$ such that, for all   $0<t\leq t_0$,
$$
d(x,\partial\{u_0>0\}) < t^{1/3} \hbox{ for any } x\in \partial\{u(\cdot,t)>0\}.
$$
\end{lem}
\begin{proof}
To prove the claim we construct suitable super- and sub-solution barriers. We begin with the former. 

\no Fix $\e>0$ and let  $\Sigma(t)$ be the 
$$
\Sigma(t) :=\{x: d(x, \Omega \setminus \Omega_0) \leq \e+t^{1/3}\}.
$$
\no Let $w^+$ and $w^-$ be respectively the solutions to   
\begin{equation*}
\begin{cases}
w^+=\Delta w^+  \ \text{ in } \  \{(x,t): x\in \Sigma(t), \ 0\leq t\leq 1\},\\[1mm]
w^+=\rho_L \ \text{ on } \ \partial \Omega \times [0,1],\\[1mm]
w^+=0 \ \text{ on }   \  \{(x,t): x\in\partial\Sigma(t), \ 0\leq t\leq1\},\\[1mm]
w^+(\cdot,0)=(u_0)_+ \ \text{ on } \  \Sigma(0),  
\end{cases}
\end{equation*}
and to the elliptic equation in \eqref{parabolic_elliptic} set in the complement of $\Sigma(t)$ with zero boundary data, and define $w:= w^+ - w^-$.

\no Since $\Sigma(t)$ has the exterior ball condition, for small $t>0$, we have 
$$
|Dw_-|(\cdot,t) \geq \e+t^{1/3} \hbox{ on } \partial\Sigma(t).
$$ 

\no On the other hand, it follows, using as barriers, if necessary,  shifted versions of heat kernel as barriers,   that the restriction of $|Dw_+|$ on $\partial\Sigma(t)$ converges exponentially fast to zero as $t \to 0$.  

\no Hence,  we can choose a sufficiently small $t_0>0$ such that, for all  $ 0< t\leq t_0,$
$$
|Dw_+|(\cdot,t) <  |Dw_-|(\cdot,t) \hbox{ on } \partial\Sigma(t). 
$$

\no It follows that  $w$ is  as a supersolution for \eqref{parabolic_elliptic}. This yields that $\partial\{u(\cdot,t)>0\}$ cannot expand faster  than order of $t^{1/3}$. 

\no Similarly one can construct a subsolution barrier, based on the fast-decreasing set 
$$
\tilde{\Sigma}(t):=\{x: d(x,\Omega_0) \geq t^{1/3}\}.
$$
\no Since the arguments are similar we omit them.
\hskip-.125in
\end{proof}

\no Using the  weak theory described above, we can prove, using arguments similar to the ones in Section~5.1 of \cite{KP}, the following theorem.
\begin{theorem}\label{unique}
Assume \eqref{g}, \eqref{lbc}, \eqref{initial_1} and \eqref{initial_2}. There exists a unique viscosity solution $u$ to  \eqref{limitproblem}. Moreover, $u=(u^*)_*$, and $u$ coincide a.e. with the regular weak solution $u$ of \eqref{parabolic_elliptic}  in $Q_T$. In particular, $b(u)$ is continuous.
\end{theorem}

\subsection*{Unbounded domain $\Omega=\R^n$}
%
It can be checked easily,  using a suitable radial barrier,  that the negative phase of any viscosity solution of \eqref{parabolic_elliptic} with initial data $u_0$ satisfying \eqref{initial_1}, \eqref{initial_2} and \eqref{cond_unbdd} is contained,  for $0\leq t\leq T$, in $B_{R_1}(0)$ with $R_1=R_1(T)$.  
\smallskip

\no Let $u_n$ be the viscosity solution of  \eqref{limitproblem} in $B_{R_1+n}(0)\times [0,T]$ with initial data $u_0\chi_{\{|x|\leq R_1+n\}}$ and boundary data  $-\nu-\frac{1}{n}$ on $\partial B_{R_1+n}(0) \times [0,T]$ respectively. It is immediate that the $u_n$'s  are monotonically  increasing with respect to $n$, and, thus,  converge pointwise to some $w:\R^n \times [0,T] \to \R$. Furthermore, the $u_n$'s  solve a uniformly parabolic equation outside of $B_{R_1}$, and, thus, are  smooth in $\{R_1 \leq |x|\leq R_1+n\}\times (0,T]$. We may, therefore, conclude  that the convergence of the  $u_n$'s  to $w$ is locally uniform in $\{|x|\geq R_1+1\}\times [0,T)$ and, moreover, that  $w$ is smooth , since it solves a parabolic equation in $\{|x|\geq R_1+1\}\times [0,T)$. 

\no Now we can use standard stability arguments about viscosity solutions to show that $w$ is the unique viscosity solution of \eqref{limitproblem} in $\{|x|\leq R_1+1\}\times [0,T]$ with the initial and boundary data $u_0$ and $w$ on $\{|x|= R_1+1\}\times [0,T]$ respectively. It follows that $w$ is a viscosity solution of \eqref{parabolic_elliptic} in $\R^n\times (0,T]$.

\begin{theorem}\label{uniqueness_unbounded}
Assume that $u_0$ satisfies \eqref{g},  \eqref{initial_1}, \eqref{initial_2} and \eqref{cond_unbdd}. Then  \eqref{parabolic_elliptic}  has a unique viscosity solution in $\R^n\times (0,T]$ with initial datum $u_0$. 
\end{theorem}
\begin{proof}
\no The barrier argument described above yields that  any viscosity solution $u$ to  \eqref{parabolic_elliptic} in $\R^n\times (0,T]$, with initial condition satisfying the assumptions in the statement and has, in addition,  compact non-positive phase,  stays positive outside of a compact set, and,  thus,  solves the parabolic equation in \eqref{parabolic_elliptic} in $\{|x|\geq R\}\times (0,T]$ for some $R>0$. 
\smallskip

\no In particular,  it follows that $u$ and $b(u)$ are smooth when $|x|$ is large, and $b(u)-\nu \in  L^\infty(0,T;L^1(\R^n))$. 

\no Moreover, $u$ is a viscosity solution of \eqref{parabolic_elliptic} in $\{|x|\leq R_1+1\}\times (0,T]$ with smooth lateral boundary data. It follows that $u$ is also the unique regular weak solution of \eqref{parabolic_elliptic} given by Theorem~\ref{existence}. 
\smallskip

\no This last claim  and the decay of $u$ for large $|x|$ imply that 
$
\partial_t b(u) \in L^1(\R^n\times (0,T)).
$
 The proof of Theorem~\ref{comparison_AL}  yields that $u$ is the unique viscosity solution of \eqref{parabolic_elliptic} with initial data $u_0$.
\end{proof}


\section{The construction of radial barriers and the proof of Theorem~\ref{thm:radial}}
\label{sec:radial}
\no We prove here Theorem~\ref{thm:radial}.  Let  $\phi$ be  a radial classical subsolution to  \eqref{parabolic_elliptic}, as defined in Definition~\ref{de:classicalSubsolution} in the domain $\{r_1\leq|x|\leq r_2\}\times [0,T]$, where $0<r_1<r_2\leq \infty$, and,   to simplify,  we  assume that the elliptic phase is contained in $B_{r_1}$. 
\smallskip

\no We recall that this means that $\phi$ is smooth in its positive and negative phase and, for some $a\in C^1([0,T];\R)$, 
$$
\left\{\begin{array}{lll}
-\Delta\phi ^-\leq G(\phi^-) &\hbox{ in }& \{\phi<0\}=\{(x,t):r_1<|x|<a(t)\},\\[1mm] 
\phi_t - \nu\Delta\phi \geq (\phi -\nu)G(0)  &\hbox{ in }& \{\phi>0\} =\{(x,t):a(t)<|x|<r_2\},\\[1mm]
0<|D\phi^+| <|D\phi^-| &\hbox{ on }& \{(x,t):|x|=a(t)\}.
\end{array}\right.
$$
\no We will perturb $\phi$ to construct subsolutions to  \eqref{eqn:um}. It is, however, important to remark that these subsolutions do not have to be, actually will not be, smooth across their interface. Discontinous subsolutions to the equation \eqref{pme} can be defined in the viscosity sense in a fashion similar to Definition~\ref{de:viscositySubsolution}
\smallskip

\no To keep things simple, we assume that (i) we have equality instead of inequality in the elliptic and parabolic equation in each phase of $\phi$, and (ii) 
the  domain is $\R^n\times [0,\infty)$ and the positive (elliptic) phase of $\phi$ is $\{|x|<a(t)\}$, that is we assume that there is no $r_1$ and $r_2 = \infty$. A minor modification of the arguments presented below yields  the general case.

\medskip

\no We introduce the notation   
$$p_0:=\phi_-\hbox{ and } \rho_0:= 1+\nu^{-1}\phi_+,
$$
and note that, for each $t\in (0,T]$,  $p_0(\cdot,t)$ solves 
\begin{equation}\label{p_0}
\left\{\begin{array}{lll}
-\Delta p_0 =G(p_0) &\hbox{ in }& \{ |x| < a(t)\},\\
p_0 = 0 & \hbox{ on }& \{|x|= a(t)\},
\end{array}\right.
\end{equation}
while  $\rho_0$ satisfies 
\begin{equation}\label{rho_0}
\rho_t - \nu\Delta \rho = \rho G(0) \hbox { in } \{(x,t): |x|>a(t)\}
\end{equation}
with the free boundary condition
\begin{equation}\label{boundary}
|D\phi^-|=\nu|D\rho_0|>|Dp_0| = |D\phi^+| \quad \hbox{ on } \{(x,t): |x|=a(t)\}.
\end{equation}

\no Our first claim is:

\begin{lem}\label{lemma:2}
There exist subsolutions $u_m$ to \eqref{eqn:um} which converge, as $m\to\infty$, uniformly to $\phi$, and, hence,  there exist viscosity subsolutions $\rho_m$ of \eqref{pme} such that, as $m\to \infty$,  the $\rho_m$'s   and the associated  pressure variables $p_m$'s  converge uniformly  to $\rho_0$ and $p_0$ respectively.
\end{lem}

\begin{proof}

Consider the one-to one functions
\begin{equation}\label{11}
\Phi(\rho):=-\rho^m + \nu(1-\rho_m)   \ \text{and} \  \Psi(p):=\Phi((\frac{m-1}{m} \rho)^{1/m-1})
\end{equation}
and recall from the introduction that 
\begin{equation}\label{above}
u_m=\Phi(\rho_m)= \Psi(p_m).
\end{equation}

%
\no It is then immediate that, if $p_m$ is bounded, then  
\begin{equation}\label{comparison}
u_m =-p_m +O(m^{-1}\ln m) \ \text { if } \ p_m \geq m^{-1}\ln m.
\end{equation}

\no The heuristics  in the introduction suggest as a possible way to prove the claim  to perturb $p_0$ and $\nu(\rho_0-1)$ from the regions $\{p_0>0\} = \{|x|<a(t)\}$ and $\{\rho_0<1\} = \{|x|>a(t)\}$ and patch them together to construct a supersolution $u_m$ of \eqref{eqn:um}, which then would imply that $\rho:=\Phi^{-1}(u)$ is a supersolution of \eqref{pme}.

\smallskip

\no The barrier will be of the form
$$
u_m(x,t) = u_m(|x|,t) =  \left\{\begin{array}{lll}
-u_{m,1}(|x|,t) &\hbox{ in }& \{|x|<a(t)\},\\[1mm]
                       u_{m,2}(|x|,t) &\hbox{ in }& \{|x|>a(t)\},\\
                              \end{array}\right.
$$

\no where $-u_{m,1}$ and $u_{m,2}$ are classical subsolutions to \eqref{eqn:um} in their respective regions and at $r=a(t)$ satisfy $|Du_1| <|Du_2|$. This inequality, which  follows from \eqref{boundary}, will prevent any smooth functions from crossing $u_m$ from below at $\{r=a(t)\}$. It follows  that $u_m$ is a viscosity supersolution of \eqref{eqn:um} in the entire domain.

\smallskip

\no We begin with  $u_{1,m}$. Let   $\tilde{u}_m(\cdot,t)$  solve 
$$
-\Delta \tilde{u}_m (\cdot,t) = G(\tilde{p}_m) + f_m(|x|) \ \hbox{ in } \{|x|<a(t)\} \ \hbox{ and } \  \tilde{u}_m(a(t),t) =0,
$$
where  
\begin{equation}\label{supersolution}
f_m(r):=A_0\nu^{-1}\chi_{\{p_0 \leq m^{-1/3}\}}(r) + m^{-1/3}  \ \text{and }  \  \tilde{p}_m:= \Psi^{-1}(-\tilde{u}_m),
\end{equation}
with  $A_0>0$ an   independent of $m$ sufficiently large constant to be determined below, and observe that, in view of the form of the equation above,
$\tilde{u}_m$ is spatially radial and
$$\tilde{u}_m(x,t)=\tilde{u}_m(\frac{x}{a(t)},1).$$

\smallskip

\no Note that,   as $m\to\infty$, $f\to 0$  in $L^1([0,a(t)])$, since $|Dp_0|(a(t))\neq 0$. Using that  $\tilde{p}$ is radial,  we find that, as $m\to\infty$,  $\tilde{u} \to p_0$ in the  $C^1$-norm in $\{|x|\leq a(t)\}$. Since $|Dp_0|>0$ and $\tilde{u}=p_0=0$ at $|x|=a(t)$, it then follows that
$$
p_0(x)=\tilde{u}(x) +o(|x-a(t)|)\hbox{ near }|x|=a(t).
$$
In particular, in view of   \eqref{comparison},  we have
\begin{equation}\label{observation}
f(r) \geq A_0\nu^{-1}\chi_{\{\tilde{u}\leq \frac{2}{3}m^{-1/3}\}}(r) + m^{-1/3} \geq A_0\nu^{-1}\chi_{\{\tilde{p} \leq \frac{1}{2} m^{-1/3}\}}(r)+m^{-1/3}.
\end{equation}

\no Next we  define
$$
u_{m,1}:= \tilde{u}_m - c_m \  \hbox{with } c_m := \Psi(\nu/m^3),
  $$  
  
\no and remark that, for any $k>1$, 
  \begin{equation}\label{est:c_m}
  \frac{1}{m}\ln m \leq c_m \leq \frac{1}{m^k}. 
  \end{equation}
  
\no The aim is  to show that, for  $u_{m,1}$, $\rho_{m,1}:=\Phi^{-1}(-u_{m,1})$ and $p_{m,1}:= \Psi^{-1}(-u_{m,1})$, 
  \begin{equation}\label{eqn:u_1}
(u_{m,1})_t - (mp_{m,1}+\nu) \Delta u_{m,1} \geq (mp_{m,1}+\nu)\rho_1 G(p_{m,1}) \ \hbox{ in } \  \{(x,t):|x|<a(t)\} ,
  \end{equation}
which yields that  $u_m=-u_{m,1}$ is  a subsolution of \eqref{eqn:um} in the region $\{(x,t):|x|<a(t)\}$.

\medskip

\no To show \eqref{eqn:u_1}, note that, since $u_{m,1}(x,t) =u_{m,1}(\frac{x}{a(t)},1)$,  it follows that,  for sufficiently large $m$,
$$
|(u_{m,1})_t| = |(\tilde{u}_m)_t |\leq |(x/a(t))' D\tilde{u}_m|  \leq C|Dp_0|=O(1) \ \hbox{ in } \ \{(x,t):|x|<a(t)\}.
$$

\no Also note that, in view of  \eqref{est:c_m},  the difference $\tilde{p}_m - p_{m,1}$ is  of order of $c_m$ if $p_m \geq m^{-1/2}$.
In particular if $p_{m,1}\leq \frac{1}{10}m^{-1/3}$, then $\tilde{p}_m\leq \frac{1}{2}m^{-1/3}$ and,  therefore,  in view of  \eqref{observation}, we find 

$$
\begin{array}{lll}
-(mp_{m,1}+\nu)\Delta u_{m,1} &=& -(mp_{m,1}+\nu)\Delta \tilde{u} _m\\[1.1mm]
&\geq & (mp_{m,1}+\nu)[G(\tilde{p}_m) + m^{-1/3}] +(mp_{m,1}+\nu)(f(r)-m^{-1/3}) \\[1.1mm]
&\geq& (mp_{m,1}+\nu)G(\tilde{p}_m) + m^{2/3}p_1 + (mp_{m,1}+\nu)(f(r)-m^{-1/3}),\\[1.1mm]
&\geq& (mp_{m,1}+\nu)p_{m,1}^{1/m-1}G(p_{m,1}) -(u_{m,1})_t,
\end{array}
$$

\no provided that $m$ and $A_0$ are sufficienty large. 
In particular  the second inequality follows from the fact that,  in view of \eqref{observation}, $\nu f_m(\tilde{p}_m) >A_0$ if $\tilde{p}_m \leq \frac{1}{3} m^{-1/3}$,
while,  otherwise,  $p_{m,1}\geq \frac{1}{10}m^{-1/3}$ and, thus,  $mp_{m,1}\geq \frac{1}{10}m^{2/3}\to \infty$ as $m\to \infty$. 
\smallskip

\no It follows that  \eqref{eqn:u_1} holds in $\{(x,t):|x|<a(t)\}$.

\smallskip

\no We next  define $u_{m,2}:= \Phi(\hat{\rho}_m)$, where  $\hat{\rho}_m$ is  a perturbation of $\rho_0$ solving
$$
\hat{\rho}_{m,t} - \nu\Delta \hat{\rho}_m = \hat{\rho}_m G(0) -m^{-1/2}\quad \hbox { in } \Sigma:=\{(x,t):|x|>a(t), t>0\}
$$
with $$\hat{\rho}_m(\cdot,0) = \rho_0(\cdot,0)-1 + \Phi^{-1}(c_m)   \  \  \text{ and} \  \ \Phi(\hat{\rho}_m)=c_m \ \text{ on} \  \{|x|=a(t)\}.$$ 

\no Note that 
$$
\Phi^{-1}(c_m) \sim (\frac{1}{m}\ln m)^{1/m-1}\to 1^-\hbox{ as }m\to \infty,
$$
therefore,  
$$
\hat{\rho}_m = (\frac{\nu}{m^3})^{1/m-1} \sim 1 \hbox{ on } \{|x|=a(t)\}.
$$
\no  Thus the $\hat{\rho}_m$'s converge to $\rho_0$ in the $C_{x,t}^{2,1}-$ norm in $\Sigma$ and in $C_x^1$ up to the boundary. Hence, since   $|D\rho_0|(a(t),t)> 0$ and $\rho_0$ is strictly away from $1$ in $\Sigma$,   for sufficiently large $m$  we have $\hat{\rho}_m^{m-1} \leq \nu/m^3 $, and, thus, 

  \begin{equation}\label{computation}
  \Delta(\hat{\rho}_m^m) = m\hat{\rho}_m^{m-1}\Delta \hat{\rho}_m + m(m-1)\hat{\rho}_m^{m-2}|D\hat{\rho}_m|^2 \sim \frac{1}{m^2}\Delta \hat{\rho}_m + \frac{\nu}{m}|D \hat{\rho}_m|^2 = O(m^{-1}).
  \end{equation}
\no Then, for sufficiently large $m$,  
$$
\hat{\rho}_{m,t} - \Delta (\hat{\rho}_m^m) - \nu\Delta \hat{\rho}_m \leq \hat{\rho}_{m,t} - \nu\Delta\hat{\rho}_m +O(m^{-1}) \leq \hat{\rho}_m G(0) \quad\hbox{ in } \Sigma,
$$

\no and, therefore,  $u_{m,2} $ is a supersolution to  \eqref{eqn:um} in $\Sigma$.

\smallskip

\no  To conclude we check that the spatial gradients of $u_{m,1}$ and $u_{m,2}$ are ordered in the right order  at the patching location $|x|=a(t)$.  This follows since, in view of  \eqref{boundary} and the $C^1$-convergence of $\tilde{u}_m$ to $p_0$ in $\{|x|\leq a(t)\}$,
$$
|D u_{m,1}| = |D\tilde{u}_m|=(1+o(1))|D p_0| < (1-o(1)) \nu |D \rho_0| < |D u_{m,2}|  \quad\hbox{ on } \{|x|=a(t)\},
$$
where the second inequality is due to \eqref{boundary} and the last inequality follows from the fact that 
$$
|D\rho_m^m| \leq m^{-2}|D\rho_m|\quad\hbox{ when }\rho^{m-1} = \nu/m^3.
$$
\end{proof}

\smallskip

\no Next we choose a radially monotone classical supersolution $\phi$ of \eqref{parabolic_elliptic} and define $\rho_0$ and $p_0$ as before. Note that in this case we have
\begin{equation}\label{boundary2}
|Du^+|=\nu|D\rho_0|<|Dp_0| = |Du^-| \quad \hbox{ on }\{(x,t):|x|=a(t)\}.
\end{equation}

\begin{lem}\label{lemma:2b}
There exists a supersolution $\rho_m$ of \eqref{pme} such that, as $m\to \infty$,  $\rho_m$ and  $p_m$  converge  uniformly  to $\rho_0$ and  $p_0$ respectively.
\end{lem}

\begin{proof}

The argument parallels that of the proof of the previous lemma. We define the barrier $u_m(x,t)=u_m(|x|,t)$  as 
$$
u_m(r,t) = \left\{\begin{array}{lll}
-u_{m,1}(r,t)  &\hbox{ in }& \{r<a(t)\},\\
u_{m,2}(r,t),   &\hbox{ in } &\{r>a(t)\},
                              \end{array}\right.
$$
with $-u_{m,1}$ and $u_{m,2}$ solving, in the respective region,
\begin{equation}\label{eqn:u2}
u_{m,t} - (mp_m+\nu)\Delta u_m \geq -\rho(mp_m+\nu)G(p_m),
\end{equation}
and $$|Du_{m,1}|> |Du_{m,2}| \ \text{on}  \ r=a(t).$$ 
\no This  makes $u_m$ a (nonsmooth) supersolution of \eqref{eqn:um}.

\smallskip

\no We define $u_{m,1} = \tilde{u}_m-c_m$, where $c_m= \Psi(\nu/m^3)$ is as before and $\tilde{u}_m$ solves 
$$
-\Delta \tilde{u}_m= G(p_m) - f_m(x)  \ \hbox{ in } \{|x|<a(t)\}, \  \hbox{ with } \ u_{m,1}(a(t)) = 0,
$$
with $f_m $  as in \eqref{supersolution} in $\{|x|<a(t)\}$. 
\smallskip

\no It follows that $\tilde{u}_m(\cdot,t)$ converges to $p_0(\cdot,t)$ in the $C_x^1-$norm in $\{|x|<a(t)\}$. Moreover, since $u_{m,1,t} = O(1)$, a  straightforward computation, similar to the supersolution case, yields that
$$
u_{m,1,t} - (mp_{m,1}+\nu)\Delta u_{m,1} \leq \rho_{m,1}(mp_{m,1}+\nu)G(p_{m,1}),
$$
and thus \eqref{eqn:u2} is satisfied.

\smallskip

\no  Next  let $u_{m,2}:= \Phi(\hat{\rho}_m)$, where $\hat{\rho}_m$ is  a perturbation of $\rho_0$ solving
$$
\hat{\rho}_{m,t} - \nu\Delta \hat{\rho}_m = \hat{\rho}_m G(0) +m^{-1/2} \hbox { in } \Sigma:=\{(x,t):|x|>a(t), t>0\},
$$
with $\hat{\rho}_m(\cdot,0) = \rho_0(\cdot,0) -1+\Phi^{-1}(c_m)$ and $\Phi(\hat{\rho}_m)=c_m$ on $\{|x|=a(t)\}$. 

\smallskip

\no Observe that, as $m\to \infty$,   $\hat{\rho}_m$ converges to $\rho_0$ in the $C_{x,t}^{2,1}-$norm in $\Sigma$ and in the $C_x^1-$norm up to the boundary, and, since $|D\rho_0|(a(t),t)>0$ and $\rho_0 <1$ in $\Sigma$, for sufficiently large $m$, we find
  $$
\hat{\rho}_m^{m-1} \leq  \hat{\rho}_m^{m-1}(a(t)) \leq \Psi^{-1}(c_m) =\nu/m^3  \quad\hbox{ in } \Sigma.
$$ 
\no Thus it follows that 
  \begin{equation}\label{computation2}
  \Delta(\hat{\rho}_m^m) = m\hat{\rho}_m^{m-1}\Delta \hat{\rho}_m + m(m-1)\hat{\rho}_m^{m-2}| D\hat{\rho}_m|^2 \sim \frac{1}{m^2}\Delta \hat{\rho}_m + \frac{\nu}{m}|D \hat{\rho}_m|^2 = O(m^{-1}),
  \end{equation}
and,   for sufficiently large $m$, $u_{m,2}$ is a supersolution of  \eqref{eqn:um}, in view of the fact, that in this case   
$$
\hat{\rho}_{m,t} - \Delta (\hat{\rho}_m^m) - \nu\Delta \hat{\rho}_m \geq \hat{\rho}_{m,t} - \nu\Delta\hat{\rho}_m -O(m^{-1})\geq \rho G(0) \quad\hbox{ in } \Sigma.
$$



\no Lastly, in view of  \eqref{boundary2}, we have on $\{(x,t): |x|=a(t)\}$
$$
|D u_{m,1}| (a(t),t) \geq (1-o(1))|Dp_0|(a(t),t) > (1+o(1))\nu|D\rho_0|(a(t),t)\geq |Du_{m,2}| (a(t),t),
$$
which is the correct order for $\tilde{u}_{m}$ to be a subsolution to  \eqref{eqn:um}.
 
  \end{proof}

Combining the two lemmas above yields the proof of Theorem~\ref{thm:radial}.

%


\section{Convergence in the general case}
\label{sec:general}

\no We  use the results of the previous section to show the general convergence result, that is Theorem~\ref{thm:main}. We only consider the case $\Omega=\R^n$, since the arguments for bounded $\Omega$ are similar. 

\smallskip


\no Let
\begin{equation}\label{limsup_liminf}
u:= {\limsup}^* u_m \hbox{ and } v:= {\liminf}_* u_m.
\end{equation}

\no We have:

\begin{theorem}\label{sub_super}
$u$ is a subsolution and $v$ is a supersolution to  \eqref{parabolic_elliptic}. 
\end{theorem}

\begin{proof}

We will only show that $u$ is a subsolution of \eqref{parabolic_elliptic}. The proof of the other claim is similar. 

\smallskip

\no If $u$ is not a subsolution,  there must exist  a classical strict supersolution $\phi$ of \eqref{parabolic_elliptic} in a parabolic cylinder $\mathcal{C}:=B_r(x_0)\times [t_1,t_0]$ such that $ u <\phi$ on the parabolic boundary of $\mathcal{C}$ and $u$ crosses $\phi$ from below for the first time at $t=t_0$. Let $(x_0,t_0)$ be where $u-\phi$ take its nonnegative maximum. 

\smallskip

\no Next we perturb $\phi$ so that it becomes slightly smaller and thus  $u-\phi$ achieves  a positive maximum in $\mathcal{C}$. For instance, we may  replace $\phi$ by the inf-convolution  
$$
\tilde{\phi}(x,t):= \inf_{\{ y\in\bar \Omega: |x-y|\leq r\}} \phi(y,t)
$$
for sufficiently small $r>0$, and  solve the corresponding elliptic and parabolic problem in the positive and negative phase of $\phi$. Note that this perturbation preserves the supersolution condition $|D\phi^+|<|D\phi^-|$ on the zero level set. 

\medskip

\no The definition of $u$ yields a subsequence, that we still denote by $m$, along which, for sufficiently large $m$,  $u_m-\phi$ has  a positive maximum in $\mathcal{C}$ with $u_m<\phi$  on the parabolic boundary of $\mathcal{C}$.  Since $u_m$ is continuous, this means that $u_m=\phi$ at $(x_m,t_m)$ for the first time in $\mathcal{C}$ with $t_1< t_m \leq t_0$.  
\smallskip

\no Let $(y_0,s_0)$ be a limit point of $(x_m,t_m)$. We claim that 
\begin{equation}\label{claim}
(y_0,s_0)\in\partial\{\phi>0\}.
\end{equation}
 To see this, suppose $\phi(y_0,s_0)>\delta>0$. Then the same holds at $(x_m,t_m)$ for large enough  $m$. It follows  from the definition of $u_m$ that,  as $m\to\infty$, $mp_m \to 0$  uniformly. 
 This and the fact that $\phi$ is a supersolution of the parabolic equation in \eqref{parabolic_elliptic} yield a contradiction by a maximum principle argument, since, in view of \eqref{u_m},  at $(x_m,t_m)$
$$
\begin{array}{lll}
\phi_t - \nu\Delta\phi &\leq& u_{m,t} - \nu\Delta u_m \\ 
& \leq & mp_m\Delta u_m - mp_m\rho_mG(p_m)\nu\rho_m G(p_m)   \\
&\leq &-\nu\rho_m G(0) + o(1)   \\ 
&\leq & (\phi-\nu)G(0)+o(1).
\end{array}
$$
\no Note that in the third inequality of above computation we have used that \\$mp_m\Delta u_m (x_m,t_m) \leq mp_m\Delta\phi(x_m,t_m)$ and, in view of the regularity of $\phi$ in the positive phase, $\lim_{m\to \infty}mp_m\Delta\phi(x_m,t_m)=0$.  In the last inequality we have used that  $-\nu\rho_m = u_m-\nu$. 
\smallskip

\no Similar arguments  yield that the case $(y_0,s_0)\in\partial\{\phi>0\}.$

\smallskip
 

\no For the rest of the proof we consider  the subsequence of $m$'s  along which the  $(x_m,t_m)$'s  converges to $(y_0,s_0).$ Since $\phi$ has $C^{2,1}-$free boundary, there exists a spatial ball $B_r(z_0)\subset \{\phi(\cdot,s_0)<0\}$ which touches the free boundary at $y_0$. 

\no Now we  approximate $\phi$ by a radial supersolution $\psi$ of \eqref{parabolic_elliptic} in the space-time domain 
$$
\Sigma:=\{x:r_1<|x-z_0| <r_2\}\times [s_0-\tau,s_0],
$$
for some $r_1<r<r_2$.  Roughly speaking the construction amounts to taking Taylor's  expansion of $\phi$ at $(y_0,s_0)$ in each phase, up to the first order in space and time, and constructing a radial function with them. 

\smallskip

\no Let $s$ be the outward normal velocity of the free boundary $\partial\{\phi>0\}$ at $(y_0,s_0)$, pick   $\e>0$ small, let $r(t):= r+a(t-s_0)$ for $a = s+\e$ so that $B_{r(t)}(z_0) \subset \{\phi(\cdot,t) <0\}$  for $t_1<t<s_0$ with  $t_1$ sufficiently close to $s_0$, and, moreover, 
\begin{equation}\label{observation1}
\phi(x,t_1) \leq -C\e(s_0-t_1) \hbox { in } B_{r(t_1)}(z_0).
\end{equation}

\no Now consider $\psi=\psi(x,t)$ in $\Sigma$ such that
$$
\left\{\begin{array}{lll}
\psi(\cdot,t) = 0 &\hbox{ on }& \{|x-z_0|=r(t)\}, \\
-\Delta\psi_- = G(\psi_-)&\hbox{ in } &\{r_1\leq |x-z_0| \leq r(t)\},\\
\psi_{+,t} - \nu\Delta\psi_+ =(\psi_+-\nu)G(0)&\hbox{ in }& \{r(t) \leq |x-z_0|\leq r_2\},\\ 
\psi = \phi_+(z_0+ r_2\nu,t)+\e (r_2-r(t)) &\hbox{ on }& \{|x-z_0|=r_2\},\\
\psi = \phi_-(z_0+ r_1\nu,t) -\e (r(t)-r_1)&\hbox{ on }& \{|x-z_0| = r_1\}.\\
\end{array}\right.
$$
Note that \eqref{observation} and the maximum principle for the elliptic equation in \eqref{parabolic_elliptic} yield that  $\phi(\cdot,t_1) <\psi(\cdot,t_1)$ in $B_{r(t_1)}(z_0) \setminus B_{r_1}(z_0)$. 
\smallskip

\no The initial datum  for $\psi$ in the parabolic phase at $t=t_1$ is given, for  $r = |x-z_0|$, by 
$$
(\phi(z_0+r_2\nu,t_1)+\e (r_2-r(t)))(r_2-r(t_1))^{-1}(r-r(t_1)), 
$$
so that, if $r_2$ is sufficiently close to $r(t_1)$,  $\phi(x,t_1) \leq \psi(x,t_1)$ on $B_{r_2}(z_0)\setminus B_{r(t_1)}(z_0)$. 

\no Also note that, since $\phi$ is a smooth strict supersolution of \eqref{parabolic_elliptic} with nonzero gradient at $(x_0,t_0)$,  if   $r_1$ and $r_2$ are sufficiently close to $r$,  then $|D\psi_+|$ and $|D\psi_-|$ are close to $\phi$ at $(y_0,s_0)$ up to  order $\e$.   

\no Thus we can conclude that $\psi$ is a supersolution of \eqref{parabolic_elliptic} in $\Sigma$ if $(t_1,r_1,r_2,\e)$ is sufficiently close to $(s_0,r,r,0)$.

\smallskip

\no Finally note that  $u_m<\phi<\psi $ on the parabolic boundary of $\Sigma$ with $\phi=\psi$ at $(y_0,s_0)$. 

\smallskip

\no The construction in the previous section  yields  a supersolution $\psi_m$ of \eqref{eqn:um} which  converges uniformly to $\psi$ as $m\to\infty$. By shifting, if needed,  $\psi$ down by a small amount, yields that  $u_m$ crosses $\psi$ from below in the interior of $\Sigma$,  a contradiction to the comparison principle of \eqref{pme}.

\end{proof}





\no Next we show that $u$ and $v$ coincide initiallly. For this we need \eqref{initial_intro} and  \eqref{initial_intro2}.

\begin{lem}\label{initial}
Assume \eqref{initial_intro} and  \eqref{initial_intro2}. Then, as $t\to 0^+$,
\begin{itemize}
\item[(a)] the sets $\partial\{u(\cdot,t)>0\}$ and $\partial\{v(\cdot,t)>0\}$
converge uniformly to $\Omega_0$  in the Hausdorff distance, and
\item[(b)] both $u(\cdot,t)$ and $v(\cdot,t)$ converge in $\Omega_0$ to 
the harmonic function $w$ in $\Omega_0$ with the zero boundary data, while 
outside of $\Omega_0$,  they converge to the original initial data $(u_0)_+:= \nu(1-\rho_0).$
\end{itemize}
\end{lem}

\begin{proof}

\no Assumptions \eqref{initial_intro} and \eqref{initial_intro2} imply that, as $m\to \infty$, there exist a sequence of positive numbers $\e_m\to 0$ such that 
\begin{equation}\label{initial_order}
\Omega_0=\{\rho_0= 1\}\subset \{u_m(\cdot,0) \leq 0\} \subset \{\rho_0>1-\e_m\},
\end{equation}
and, in particular,  the sets $\{u_m (\cdot,0)\leq 0\}$'s converge to $\Omega_0$ in the Hausdorff distance.
\smallskip

\no We prove (a) first.  Fix $\e>0$, $x_0\in\Omega_\e:= \{x: d(x,\Omega_0) \leq \e\}$ and let  $B_\e(y_0)$ be the ball which touches $\partial\Omega_0$ at $x_0$ from outside of $\Omega_0$.  With  $a(t):= \e/2 -Mt$, we consider the radial function $\phi$ given by 
$$
\left\{\begin{array}{lll}
-\Delta\phi_- = G(\phi_-) &\hbox{ in }& \{x: a(t)\leq |x-y_0| \leq (C+1)\e\},\\ 
\phi_-(\cdot,t) = M_0 &\hbox{ on }& \{|x-y_0|= (C+1)\e\},\\ 
\phi(\cdot,t)=0 &
\hbox{ on } &\{ |x-y_0|=a(t)\},\\ 
\phi_t - \nu\Delta\phi = (\phi-\nu)G(0) &\hbox{ in } &\{x: |x-y_0| \leq a(t)\},\\
\phi (\cdot, 0)=c_0|x-y_0|^2 - c_0\e^2/4 &\hbox{ in} & \text{ the parabolic phase}, 
\end{array}\right.
$$
with $c_0$ is chosen small enough so that $\phi (\cdot, 0)$ sits below $u_m(\cdot,0)$ in $B_{\e/2}(y_0)$; note that the  support of $\phi (\cdot, 0)$ equals $B_{\e/2}(y_0)$.
 
\no  It follows from  by Lemma~\ref{appendix} that $|D\phi_+|(\cdot,a(t))$ is sufficiently large for some  large $M$ and for small enough $t$. Choosing such  $M$ and $t$
we have $|D\phi_-| <|D\phi_+|$ on $x=a(t)$, and, thus,  $\phi$ a subsolution of \eqref{parabolic_elliptic} for a small time interval $[0,t_0]$. Using the approximating solutions for the $u_m$-equation constructed  in the previous section  as well as \eqref{initial_order},  we conclude  that $v \geq \phi$, and, hence,  $\{v(\cdot,t)<0\}$ is outside of $B_{a(t)}(y_0)$ for $0\leq t\leq t_0$. Since $x_0$ was arbitrarily and $\Omega_\e$ is compact, we find  
$$
\{v(\cdot,t)<0\} \subset \Omega_\e \hbox{ for } t\leq t_0=t_0(\e). 
$$
Letting $\e \to 0$ gives  
\begin{equation}\label{upper_bd}
\lim \sup_{t\to 0^+}\{v(\cdot,t)<0\} \subset \overline{\Omega_0}.
\end{equation}

\no Next we show that  
\begin{equation}\label{lower_bd}
\overline{\Omega_0} \subset \lim\inf_{t\to 0^+} \{u(\cdot,t)< 0\}. 
\end{equation}

\no Fix $B_{2r}(x_0) \subset \Omega_0$, let $f(x,t):= t(|x-x_0|^2 -r^2)$ and denote by $p_f$ the pressure function corresponding to $f$. Then, as long as $r$ is sufficiently small, 
$$
f_t - (mp_f + \nu) \Delta f = (|x-x_0|^2 -r^2)  +2nA(mp_f+\nu)t|x|^2 \geq -\nu G(p_f) \hbox{ in } B_r(x_0) \times [0, 1],
$$

\no Let  
$$
g(x,t):= f(x,t) \chi_{|x-x_0|\leq r}+ h(x,t)\chi_{|x-x_0|\geq r},
$$ 

where $h(x,t)$ solves a heat equation in $(\R^n-B_r(x_0))\times [0,1]$ with initial datum $(u_0)_+$ and zero boundary condition on $\partial B_r(x_0)$. Since $(u_0)_+ = 0$ in $B_{2r}(x_0)$, it follows from the heat kernel estimate that,  for small time, 
$$
|Dh|(\cdot,t) \leq t^{-3/2} e^{-r^2/t} \leq |Df|(\cdot,t) \hbox{ on }\partial B_r(x_0).
$$
\smallskip

\no Thus $g$ is a supersolution of \eqref{eqn:um} in $\R^n\times (0,t_0)$ for sufficiently small $t_0$ depending on $r$, and, hence,  for $0\leq t\leq t_0$,   $B_r(x_0) \subset \{u(\cdot,t)<0\}$. Since $r$ is arbitrarily,  \eqref{lower_bd} follows.

\vspace{10pt}

\no In view of the facts that  $v\leq u$ and  $\{u(\cdot,t)<0\} \subset \{v(\cdot,t)< 0\}$, the two inclusions above yield (a).

\vspace{10pt}

\no Now we prove (b).  Fix $\e>0$ and define  
$$\Omega_f := \{x: d(x, \R^n \setminus \Omega_0) > \e \}\hbox{ and }\Omega_g:=\{x:d(x,\Omega_0) \leq \e\}. 
$$
In view of  (a),  there exist $\delta = \delta(\e)>0,  t_0=t_0(\e)>0$ and $M$ such that for $m>M$ and $0\leq t\leq t_0$ the following holds:  $u_m \geq \delta$ on $\partial\Omega_g$, $u_m \leq -\delta$ in $\Omega_f$, and moreover  
\begin{equation*}
- \Delta f = G(f)-\e \hbox{ in } \Omega_f  \ \text{ and} \  f=\delta\hbox{ on }\partial\Omega_f,
\end{equation*}
and
\begin{equation*}
 -\Delta g = G(g)+\e\hbox{ in }\Omega_g  \ \text{ and} \  g=\delta\hbox{ on }\partial\Omega_g.
\end{equation*}

%
\no Let 
$$
\phi(x,t):= -a(t)f(x) \hbox{ and } \psi(x,t) := -b(t) g(x),
$$ 
where
$$
  a(t):= \min [-\delta e^{m\e t} ,1] \hbox{ and } b(t) := \max[\delta^{-1}e^{-m\e\delta t},1].
$$
\no Direct calculations then yield that, for sufficiently large $m$ and   any choice of $t_0>0$, $\phi$ and $\psi$  are respectively supersolution and subsoluton to the  $u_m$-equation in $\Omega_f\times (0,t_0]$ and $\Omega_g \times (0,t_0]$.   Thus the comparison principle for the $u_m$-equation and the choice of $\delta$ and $t_0$ yield
$$
\psi\leq u_m \hbox{ in } \Omega_g \times [0,t_0] \hbox{ and } u_m \leq \phi\hbox{ in }\Omega_f \times [0,t_0].
$$
\no Letting $m\to\infty$ and using arbitrarily small $\e>0$,  we conclude that the  $(u_m)_-$'s converge  uniformly  to the solution of the elliptic equation $\Omega_0$ with zero boundary data. 
\smallskip

\no Similar arguments apply to  $(u_m)_+$.
\end{proof}

\smallskip

\no To derive the convergence result for $u_m$'s  using the previous  lemmata, we need to first show  that, as $m\to\infty$, the $\nu(1-\rho_m)$'s  stay nonnegative. 

\begin{lem}\label{density_ub}
Assume \eqref{pmeboundary}, \eqref{pmeboundary2} and \eqref{g}. 
Then 
$$
\limsup_{m\to\infty} p_m (x,t) \leq p_M \ \hbox{ in }  \ Q_T,
$$
and, hence, 
$$\limsup_{m\to\infty} \rho_m \leq 1 \ \hbox{ in }  \ Q_T.
$$
\end{lem}

\begin{proof}
Fix $\e>0$ and let  $C:=\max [p_m(\cdot,0),p_M+\e]$.  Then $G(p_M+\e) =-\delta <0$ for some $\delta>0$. Thus 
$$f(x,t):= \max[C-(m-1)C\delta t, \quad p_M+\e]$$ is a supersolution of \eqref{eqn:p} and $p_m \leq f$ in $Q_T$, if $p_m \leq f$ on the parabolic boundary of $Q_T$. At $t=0$ this is guaranteed from the definition of $C$, and on the lateral boundary it follows from  \eqref{pmeboundary2}.
Since $f = p_M+\e$ after $t\geq O(\frac{1}{m})$,  we conclude that 
$$
\limsup_{m\to\infty} p_m (x,t)\leq p_M+\e\hbox{ for all }(x,t)\in Q_T.
$$
The claim follows after sending $\e$ to zero. 

\end{proof}

\smallskip
\no We are now ready to show the main convergence result.

\begin{cor}\label{convergence}
Let $u$ and $v$ be as in  \eqref{limsup_liminf}. Then:
\begin{itemize}
\item[(a)] $u=v^*$ is the unique viscosity solution $w$ of \eqref{parabolic_elliptic} with the initial data $u_0$.
\item[(b)] As $m\to\infty$, the $\nu(1-\rho_m)$'s and $\rho_m$'s converge uniformly  respectively to $b(w)$ and $1-\nu^{-1}b(w)$ in $Q_T$.\\
\end{itemize}
\end{cor}
\begin{proof}
It follows from Theorem~\ref{sub_super} and Lemma~\ref{initial}  that $u$ is a subsolution and $v$ is a supersolution of \eqref{parabolic_elliptic} with initial data $u_0$. 
Let  $w$ be the unique viscosity solution of \eqref{parabolic_elliptic} with initial data $u_0$ given in Theorem~\ref{unique}. Then $u \leq w$ and $w_*\leq v$, and,  hence, 
$$
u\leq w=(w_*)^*\leq v^*.
$$
Since $v^* \leq u$ by definition, we have   $u=v^*=w$. A similar argument leads to  $v=w_*=u_*$. 

\smallskip

\no Since, in view of  Theorem~\ref{unique} we know that $b(w)$ is continuous,   $b(w)=b(w_*)$ and thus \\$b(u)=b(v)=b(w)$.  
\smallskip

\no Now we claim that 
$$
b(v)\leq {\liminf}_* \nu(1-\rho_m) \hbox{ and }{\limsup}^* \nu(1-\rho_m) \leq b(u),
$$
which will yield the uniform convergence of the $\nu(1-\rho_m)$'s. 

\smallskip
\no To prove the first inequality in the claim above, we observe that $u_m \leq \nu(1-\rho_m)$ and, in view of Lemma~\ref{density_ub},  $\limsup_{m\to\infty}\rho_m\leq 1$,  which gives    $\liminf_{m\to\infty} \nu(1-\rho_m) \geq 0$. Then 
$$
b(v) = \max[0, {\liminf}_* u_m] \leq {\liminf}_*\nu(1-\rho_m).
$$

\no To show the second inequality, suppose that, for some  $(x_0,t_0)\in Q_T$, ${\limsup}^*\nu(1-\rho_m)(x_0,t_0)>0$,  since otherwise there is nothing to show. In this case, along a subsequence of $m\to\infty$, we have ${\liminf}_*\rho_m(x_0,t_0) <c_0<1$ and, thus,  ${\liminf}_* p_m(x_0,t_0) = 0$. This yields that 
$$
u(x_0,t_0)={\limsup}^* u_m (x_0,t_0) = {\limsup}^*\nu(1-\rho_m)(x_0,t_0),
$$ and the conclusion follows.
\end{proof}


\no  The uniform convergence of $\rho_m$, or Corollary~\ref{convergence}  seem to be the best one can do without further restriction on the initial data, even for the short time interval. In particular, when $u_0$ is a (strict) supersolution of the elliptic equation in \eqref{parabolic_elliptic}, then the $u_m$'s converge to the solution of \eqref{parabolic_elliptic} with the initial data strictly smaller than $u_0$, that is one should expect,  for sufficiently large $m$, almost discontinuous decrease of the pressure profile at the initial time for $p_m$. 

\section{Finite speed of propagation}
\label{sec:last}
\no We derive some properties of  the speed of propagation for both the negative and positive phases under certain assumptions. We know that the negative phase can nucleate in the interior of the positive phase, when $u$ decreases from a positive value to become zero in finite time.  We show here a quantitative version of this phenomena.  Indeed when $u$ becomes very small in an open neighborhood in the positive phase, then it turns into the negative phase in short amount of time. 

\begin{lem} 
Suppose $u(\cdot,t_0)\leq \e $ in $\bar{B}_r(x_0)$ with $0\leq \e<r^2$. There exists a dimensional constant  $C_n$ such that  
$$
 u(\cdot,t)<0 \hbox{ in } B_{r/4}(x_0)  \ \hbox{ for } \  t \in[ t_0+\frac{C_n}{\nu G(0)} \e, t_0+r^2].
$$

\no In particular, the interior of  $\{u_*(\cdot,t) \leq 0\}$ coincides with $\{u_*(\cdot,t)<0\}$.
\end{lem}

\begin{proof}

Note that any solution to  the parabolic equation 
$$
w_t -\nu \Delta w = (w-\nu)G(0) \hbox{ in } Q_T,
$$
which stays positive, serves as a supersolution to  \eqref{parabolic_elliptic}. This fact and the parabolic Harnack inequality yield 
\begin{equation}\label{harnack}
u \leq C\e \hbox{ for a dimensional constant } C>0 \hbox{ in } \Sigma:= B_{r/2}(x_0)\times [t_0,t_0+r^2].
\end{equation} 
\smallskip
\no Let $f(x,t) := C\e + \frac{C\e}{r^2}|x-x_0|^2 - \frac{1}{2}\nu G(0)(t-t_0)$. It is immediate that  $f$  is a supersolution to  \eqref{parabolic_elliptic} and thus $u\leq f$. In particular, 
$$
u<0 \hbox{ in } B_{r/4}(x_0)\times [t_1,t_0+r^2]  \  \hbox{ where } \ t_1:= t_0+ \frac{4C\e}{\nu G(0)}. 
$$
\end{proof}

\no The following two local estimates quantify the finite propagation property.  First we show that the negative phase (the tumor) does not shrink too fast over time. Note that it can shrink at least temporarily, when the positive part of $u$ has steep growth near the interface.  In terms of the original density variable $\rho$ this means that its  is close to zero near the tumor boundary.

\begin{lem}\label{finite_shrink}[Finite-speed Shrinkage]
Suppose $u\leq 1$ in $Q_T$. The negative phase $\{u(\cdot ,t)<0\}$ shrinks at most by $t^{2/5}$ over a time period $t$. More precisely,  there exits $c_0>0$ depending only on $n$, such that, whenever  $B_r(x_0) \subset\{u(\cdot,t_0)<0\}$ with $0<r<c_0$ and  $t_0\geq 0$, then
$$
B_{r/2}(x_0)\subset\{u(\cdot, t)<0\} \ \hbox{ for } \  t_0\leq t\leq t_0+ r^{5/2}.
$$
\end{lem}
\begin{proof}

Let  $h$ solve  $$-\Delta h= G(h) \ \text{ in} \  B_{r/2}(x_0) \  \text{ with } \ h=0 \ \text{on}  \  \partial B_{r/2}(x_0).$$ Then, since $G(0)>0$, there exists $C>0$ such that 
\begin{equation}\label{1}
|D h| \geq Cr \ \hbox{ on } \ \partial B_{r/2}(x_0). 
\end{equation} 
\no Let  $\psi$  be the solution to 
$$
\left\{\begin{array}{lll}
\psi_t-\nu\Delta\psi = \psi G(0) &\hbox{ in }& \Sigma:=(\R^n-B_{r/2}(x_0))\times [t_0,t_0+r^{5/2}],\\ 
\psi =0 &\hbox{ on }& \partial B_{r/2}(x_0)\times [t_0,t_0+ r^{5/2}],
\\
\psi =  \chi_{\{|\cdot -x_0|>r\}}&\text{ in}& (\R^n\setminus B_{r/2}(x_0)) \times \{0\}.
\end{array}\right.
$$
Using the properties of the heat kernel and  scaling arguments we find, for some dimensional  $C>0$,  
\begin{equation}\label{2}
|D\psi|(\cdot,t) \leq C\frac{1}{r} e^{-r^2/t} \hbox{ in }\partial B_{r/2}(x_0),
\end{equation}
and  it follows that 
$$|D\psi| \leq \frac{1}{r} e^{-1/\sqrt{r}} \ \text{ on the lateral boundary of $\Sigma$.}$$  

\no Hence, if $r$ is small, $\phi=\phi_+ -\phi_-$ defined by 
$$
\phi_+(x,t):= \psi(x,t) \hbox{ in } \Sigma, \quad \phi_-(x,t) := h(x) \hbox{ in } B_{r/2}(x)\times [t_0,t_0+r^{5/2}],
$$
is a supersolution of \eqref{parabolic_elliptic} in $\R^n\times [t_0,t_0+r^{5/2}]$ and the conclusion follows by comparing $u$ with $\phi$.

\end{proof}

\no Due to the possible nucleation of the negative phase,  it is not feasible  to  estimate the speed of expansion of the negative phase without knowing the positive density distribution near a given point. The following lemma states that, if $(x_0,t_0)$ is on the interface and  $u_+$ is nondegenerate near it, then the negative phase cannot expand too fast. 
 
\begin{lem}\label{expansion_finite}[Expansion bound]
Suppose $u(\cdot,t_0)>0$ in $B_{3r}(x_0)$ with $u(\cdot,t_0)>r$ in $B_r(x_0)$. Then,  for $r\leq r_0$ where $r_0$ depends only on $n$,  we have $u>0$ in $B_{r/2}(x_0)\times [t_0,t_0 + r^3]$.

\end{lem} 
\begin{proof}
Without loss of generality we set $x_0=0$. Comparing $u$ with a fundamental solution of the heat equation with initial data $ - \chi_{|x|>3r}$ as a subsolution to \eqref{parabolic_elliptic}, we find, for some $C>0$ which again depends only on $n$, 
$$
u^-(\cdot,t) \leq C_0r \hbox{ on } \partial B_{2r}(0) \times [t_0,t_0+r^3].
$$
\no Next we  construct a subsolution to  \eqref{parabolic_elliptic} in $B_{2r} \times [t_0,t_0+r^3]$ by solving, for  $r(t):= r-\sqrt{t-t_0}$, the initial-boundary value problem 
$$
\left\{\begin{array}{lll}
\phi_-(x,t) = C_0 r &\hbox{ on }& \partial B_{2r}(0),\\
-\Delta\phi_-(\cdot,t) = G(\phi) &\hbox{ in }& B_{2r}(0) - B_{r(t)}(0),\\
\partial_t\phi  -\nu\Delta\phi = -\nu G(0)  &\hbox{ in } &\cup_{t_0<t\leq t_0+r^3}(B_{r(t)(0)}\times \{t\}),\\ 
\phi(\cdot,t)=0 &\hbox{ on }& \partial B_{r(t)}(0), \\
\phi_+(\cdot,t_0) = \e\chi_{B_r(0)} &\hbox{ in }& (B_{2r} \setminus B_{r(t)}) \times \{0\}.
\end{array}\right.
$$
\no The parabolic Harnack inequality for the heat equation gives   $|D\phi_+|(\cdot,t) \geq C\frac{r}{\sqrt{t-t_0}}$ on $B_{r(t)}(0)$.  Thus 
$$
\tilde{\phi}(\cdot,t):=(\phi)_+ - (\phi)_-
$$
is a subsolution to  \eqref{parabolic_elliptic} in $B_{2r}(0)\times [t_0, t_0+r^3]$, since, for a dimensional constant $C$, 
$$
|D\phi_-|(r(t),t)\leq C \frac{\phi(2r,t)}{r}  \leq C_2\frac{r}{\sqrt{t-t_0}}\leq |D\phi_+|(r(t),t).
$$
We can now conclude by comparing $u$ with $\tilde{\phi}$ in $B_{2r}(0)\times [t_0,t_0+r^3].$

\end{proof}

\medskip

\appendix

\section{Discussion on the proof of the comparison principle in \cite{KP}.}

\no For fixed $r > 0$, we consider the open balls 
\begin{align*}
B_r(x,t) := \set{(y,s): \abs{x-y}^2 + \abs{t-s}^2 < r^2} \ \text{ and } \ 
B_r^n(x) := \set{y: \abs{x-y} < r},
\end{align*}
the space disk
\begin{align*}
D_r(x) := B_r^n(x) \times \set{0} = \set{(y, 0) : \abs{x-y} < r},
\end{align*}
and the flattened set
\begin{align*}
E_r(x,t)  := \set{(y,s) : \abs{x-y}^3 + \abs{t-s}^2 < r^2}.
\end{align*}
Finally, we define the domain $\Xi_r(x,t) $ that is used in the definition of regularizations of solutions by 
\begin{align*}
\Xi_r (x,t):= D_r (x,t) + E_r (x,t),
\end{align*}
where  $+$ is the Minkowski sum. Note that near its top portion at $t=r$, $\Xi(x,t)$ shrinks its spatial radius with the same rate. Finally, when $(x,t)=(0,0)$ we omit and we simply write $\Xi_r$.

\smallskip

\no We write next  a simple version of  the results stated in Lemma 3.13 and Lemma 3.16 of \cite{KP} . 


\begin{lem}\label{appendix}
(a)~Let $u \geq 0$ solve  $u_t - \nu\Delta u \leq 0$  in a parabolic neighborhood of $\overline \Xi_r$  for some $r\in (0,1]$ and assume that $u = 0$ on $\Xi_r$. There exists $\e > 0$ and 
$g(s):= \frac{M}{\e}((s-r)_+)^2 $, where $M$ is independent of $\e$, such that 
\begin{align*}
0 \leq u(x,t) \leq g(\abs{x}) \ \text{in }  \  \{|x|< r +\e\}\times [0, 0 + r].
\end{align*}
(b)~ Let $v$ satisfy $v_t - \nu\Delta v \geq -1$ in $\overline \Xi_r(\xi,s)$ and $v > 0$ in $\Xi_r(\xi,s)$ for some $(\xi, s) \in \Rn \times \R$. There exists $f \in C([0,r])$, $f(0) = 0$, $f > 0$ on $(0,r]$ and $\frac{f(s)}{s} \to \infty$ as $s \to 0+$ such that
\begin{align*}
v(x, t) \geq f(r - \abs{x}) \quad \text{for } \{x:\abs{x} < r\}\times [s + r/2, s + r].
\end{align*}
\end{lem}
\no The first part of the lemma states that,  if the support of $v$ is $\Xi$, then the spatial gradient of $v(\cdot,t)$ grows to infinity on the lateral boundary of $\Xi$ as $t$ increases to $r$. The second part implies  that, if the support of $v$ is outside of $\Xi$, then the spatial gradient of $v(\cdot,t)$ vanishes on the lateral boundary of $\Xi$ as $t$ increases to $r$. 
\smallskip

\no Now we touch on the proof of the comparison principle. Since it  parallels that of Theorem 3.1 in \cite{KP}, we only sketch it pointing out differences.

\no We consider the regularized sub- and supersolutions
$$
Z(x,t):= \sup_{\Xi_r(x,t)} u, \quad W(x,t):= \inf_{\Xi_r(x,t)} v,
$$
\no and argue by contradiction assuming  that there is a finite first crossing time $t_0$ defined by
$$
t_0:= \sup\{\tau : Z(\cdot,t) < W(\cdot,t) \hbox{ for } 0\leq t\leq \tau\}.
$$

\no It follows  that there is a contact point at $t=t_0$ between the free boundary of $Z$ and $W$.

\begin{lem}\label{contact} We have:
$$
\{Z(\cdot,t_0)\geq 0\} \cap \{W(\cdot,t_0)\leq 0\} = \partial\{Z(\cdot,t_0)\geq 0 \} \cap \partial\{W(\cdot,t_0)> 0\}.
$$

\end{lem}
\begin{proof}
First note that, in view of Lemma~\ref{finite_shrink},  $\{u(\cdot,t_0)\geq 0\}$ can not  expand discontinuously.  The  definition of $Z$ then yields that,  if $x_0$ lies in the interior of $\{Z(\cdot,t_0)\geq 0\}$,  then $Z \geq 0$ in $B_r(x_0)\times [t_0-t,t_0]$ for some sufficiently small $r$ and $t$. It then follows from the definition of $t_0$ and the strong maximum principle for the heat equation  that  $\{W(\cdot,t_0)>0\}$ in the interior of $\{Z(\cdot,t_0)\geq 0\}$ and the claim follows. 
\end{proof}
\no The next lemma corresponds to Lemma 3.19 in \cite{KP}. Once it is shown, the rest of the proof is the same as in \cite{KP}.  To state it and sketch the proof, it is necessary to remind the reader  the geometric properties of $Z$ and $W$ as in \cite{KP}. By definition, for any $y_0\in\partial\{Z(\cdot,t_0)\geq 0\}$ there exists a ``cylinder" $\Xi_Z = \Xi_r(z_1,s_1)$ in $\{Z\geq 0\}$ such that $(y_0,t_0)$ lies on the lateral boundary of  $\Xi_Z$. On the other hand, for any $y_0\partial\{W(\cdot,t_0)<0\}$ there exists an exterior ``cylinder" $\Xi_W = \Xi_r(z_2,s_2)$ in $\{W< 0\}$ such that $(y_0,t_0)$ lies on the lateral boundary of $\Xi_W$. 

\begin{lem}\label{finite_prop}
Fix $y_0\in\partial\{Z(\cdot,t_0)\geq 0\}\cap \partial\{W(\cdot,t_0)\geq 0\}$, and let $\Xi_Z$ and $\Xi_W$ be as given above. Then 
$$
t_0-r<s_1,s_2 <t_0+r.
$$
\end{lem}

\begin{proof}
It follows from the definitions that  $\{W(\cdot,t_0)\leq 0\}$ has an interior ball of radius $r$ which touches $y_0$ on its boundary. This fact and Lemma~\ref{finite_shrink} yield  that $s_1<t_0+r$. Since $Z\leq W$ prior to $t=t_0$, we also have $s_2>t_0-r$. If we show that $s_2 <t_0+r$,  then the ordering between $Z$ and $W$ implies  that $s_1 >t_0-r$ and, hence, the conclusion.

\smallskip 
\no Since $s_1<t_0+r$, $\Sigma:=\Xi_Z \cap\{t<t_0\}$ is nonempty and $Z\geq 0$ in $\Sigma$. Let $B:= \Xi_Z\cap\{t=t_0\}$.  The strong maximum principle as well as  Hopf's lemma imply that  $W(\cdot,t_0)>0 \hbox{ in } B$ with nonzero gradient on the boundary of $B$. It then follows from Lemma~\ref{expansion_finite} that the negative phase of $W$ cannot expand too fast, that is $s_2<t_0+r$.
\end{proof}

\end{document}